\documentclass[letterpaper,11pt]{article}

\usepackage{amsmath,amsthm,bbm,amssymb,latexsym,mathrsfs,amsfonts,titlesec,graphicx,tikz-cd,enumerate}

\RequirePackage[left=1in,right=1in,top=1in,bottom=1in]{geometry}

\titleformat{\subsection}[runin]
  {\normalfont\bfseries}{\thesubsection}{1em}{}

\newcommand{\G}{\mathbb{G}}
\newcommand{\A}{\mathbb{A}}

\numberwithin{equation}{subsection}

\newtheorem{theorem}{Theorem}[subsection]
\newtheorem{lemma}[theorem]{Lemma}
\newtheorem{conjecture}[theorem]{Conjecture}
\newtheorem{corollary}{Corollary}[theorem]
\newtheorem{proposition}[theorem]{Proposition}
\theoremstyle{definition}
\newtheorem{definition}[theorem]{Definition}
\newtheorem{example}[theorem]{Example}

\DeclareMathOperator{\hol}{hol}

\DeclareMathOperator{\Kost}{Kost}
\DeclareMathOperator{\Hom}{Hom}
\DeclareMathOperator{\res}{res}

\DeclareMathOperator{\Ext}{Ext}
\DeclareMathOperator{\gr}{gr}
\DeclareMathOperator{\id}{id}
\DeclareMathOperator{\IC}{IC}
\DeclareMathOperator{\End}{End}
\DeclareMathOperator{\Sym}{Sym}

\DeclareMathOperator{\Spec}{Spec}

\DeclareMathOperator{\RHom}{RHom}
\DeclareMathOperator{\Eis}{Eis}
\DeclareMathOperator{\ext}{ext}

\DeclareMathOperator{\triv}{triv}

\DeclareMathOperator{\colim}{colim}

\DeclareMathOperator{\std}{std}

\DeclareMathOperator{\loc}{loc}
\DeclareMathOperator{\rsd}{rsd}

\DeclareMathOperator{\Bun}{Bun}
\DeclareMathOperator{\ev}{ev}
\DeclareMathOperator{\pos}{pos}

\DeclareMathOperator{\disj}{disj}

\DeclareMathOperator{\pt}{pt \!}
\DeclareMathOperator{\add}{add}
\DeclareMathOperator{\Av}{Av}
\DeclareMathOperator{\sgn}{sgn}

\begin{document}

\title{Nearby cycles of Whittaker sheaves on Drinfeld's compactification}
\author{Justin Campbell}
\maketitle

\begin{abstract}

In this article we study the perverse sheaf on Drinfeld's compactification obtained by applying the geometric Jacquet functor (alias nearby cycles) to a nondegenerate Whittaker sheaf. Namely, we describe its restrictions along the defect stratification in terms of the Langlands dual Lie algebra, in particular showing that this nearby cycles sheaf is tilting. We also describe the subquotients of the monodromy filtration using the Picard-Lefschetz oscillators introduced by S. Schieder, including a proof that the subquotients are semisimple without reference to any theory of weights. This argument instead relies on the action of the Langlands dual Lie algebra on compactified Eisenstein series (or what is essentially the same, cohomology of quasimaps into the flag variety) constructed by Feigin, Finkelberg, Kuznetsov, and Mirkovi\'{c}.

\end{abstract}

\section{Introduction}

\subsection{} In this article we study the degeneration of a Whittaker sheaf on Drinfeld's compactification to an object of the principal series category (for us ``sheaf" will mean D-module, but our results carry over \emph{mutatis mutandi} to $\ell$-adic sheaves in characteristic $p > 0$ for $\ell \neq p$). In \cite{AG} this degeneration is implemented by a gluing functor, meaning $!$-extension from the general locus followed by $!$-restriction to the special fiber. This operation produces a complex of sheaves, which the authors link to the constant term of Poincar\'{e} series.

Nearby cycles provide a different, although closely related, method of degeneration. Notably, if we apply nearby cycles (more precisely, the geometric Jacquet functor of \cite{ENV}) to the perverse cohomological shift of a nondegenerate Whittaker sheaf, the result is still a perverse sheaf. It comes equipped with a nilpotent endomorphism, which gives rise to the so-called monodromy filtration. We describe the restrictions of this nearby cycles sheaf to the defect strata in terms of the Langlands dual Lie algebra, showing in particular that they are perverse, i.e. ours is a geometric construction of a tilting sheaf on Drinfeld's compactification. We also describe the associated graded sheaf of nearby cycles with respect to the monodromy filtration, along with its Lefschetz $\mathfrak{sl}_2$-action, in terms of the Picard-Lefschetz oscillators, which are certain factorizable perverse sheaves with $\mathfrak{sl}_2$-action introduced in \cite{S1}.

In \cite{C1}, the author introduced a similar construction in the finite-dimensional situation of a flag variety. Namely, we showed that the nearby cycles of a one-parameter family of nondegenerate Whittaker sheaves on a flag variety is the big projective sheaf, which is isomorphic to the tilting extension of the constant perverse sheaf on the big cell. Thus in both cases taking nearby cycles of Whittaker sheaves produces tilting sheaves. We also remark that both of these constructions have the feature that they degenerate a D-module with irregular singularities to one with regular singularities.

\subsection{} Fix an algebraically closed field $k$ of characteristic zero and a smooth connected curve $X$ over $k$. Let $G$ be a connected reductive group over $k$, choose a Borel subgroup $B$ with unipotent radical $N$, and put $T := B/N$. We write $I$ for the set of vertices of the Dynkin diagram of $G$, and $Z_G$ for the center. We write $D(Y)$ for the derived category of D-modules on a variety $Y$, and $D_{\hol}(Y)$ denotes the full subcategory consisting of complexes with holonomic cohomologies.

We denote by $\overline{\Bun}_{N^{\omega}}$ Drinfeld's compactification of the moduli stack $\Bun_{N^{\omega}}$ of canonically twisted $N$-bundles. There is a canonical map \[ \ev : \Bun_{N^{\omega}} \longrightarrow \G_a, \] constructed for example in \cite{AG}.

Let $\chi$ be a nontrivial exponential D-module on $\G_a$. As in \cite{FGV}, one shows that $\ev^{\Delta}\chi$ extends cleanly to a perverse sheaf $\mathscr{W}_1$ on $\overline{\Bun}_{N^{\omega}}$, where $\ev^{\Delta}$ denotes cohomologically normalized inverse image along the smooth morphism $\ev$. The sheaf $\mathscr{W}_1$ has irregular singularities because $\chi$ does, so by ``perverse sheaf" we mean any holonomic D-module.

Choose a dominant regular cocharacter $\gamma : \mathbb{G}_m \to T$, which determines an action \[ a_{\gamma} : \mathbb{G}_m \times \overline{\Bun}_{N^{\omega}} \to \overline{\Bun}_{N^{\omega}}. \] We will denote by $\mathscr{W}$ the perverse sheaf $a^{\Delta}_{\gamma}\mathscr{W}_1$ on $\G_m \times \overline{\Bun}_{N^{\omega}}$. Consider the embeddings \[ \overline{\Bun}_{N^{\omega}} \stackrel{i}{\longrightarrow} \A^1 \times \overline{\Bun}_{N^{\omega}} \stackrel{j}{\longleftarrow} \G_m \times \overline{\Bun}_{N^{\omega}}. \] The $\G_m$-equivariant object $i^!j_!\mathscr{W}$ of $D(\overline{\Bun}_{N^{\omega}})$ is studied in \cite{AG}.

We will consider instead a closely related perverse sheaf on $\overline{\Bun}_{N^{\omega}}$, namely the nearby cycles $\Psi(\mathscr{W})$ with respect to the projection $\A^1 \times \overline{\Bun}_{N^{\omega}} \to \A^1$. This sheaf is unipotently $\G_m$-monodromic, and (the logarithm of) the monodromy endomorphism of nearby cycles agrees with the obstruction to $\G_m$-equivariance, a nilpotent endomorphism. The composition $\Psi \circ a_{\gamma}^{\Delta}$ for a $\G_m$-equivariant oen-parameter family is studied in \cite{ENV} under the name \emph{geometric Jacquet functor}.

The object $i^!j_!\mathscr{W}$ can be recovered from $\Psi(\mathscr{W})$ as the derived invariants of monodromy, and conversely $\Psi(\mathscr{W})$ is the derived coinvariants of the natural $H^{\bullet}(\G_m)$-action on $i^!j_!\mathscr{W}$ (for the latter see Definition \ref{ncconstr} and Proposition \ref{ncprop}). One can summarize by saying that nearby cycles and the ``gluing functor" $i^!j_!$ are related by Koszul duality. This motivates the study of $\Psi(\mathscr{W})$, since these gluing functors (in the extended Whittaker situation) play a role in the proof of the geometric Langlands equivalence sketched in \cite{G}. Nearby cycles has the advantage of being t-exact, and therefore $\Psi(\mathscr{W})$ is amenable to more concrete description, e.g. it has a composition series which is finite over any quasi-compact open substack of $\overline{\Bun}_{N^{\omega}}$.

We also remark that up to isomorphism $\Psi(\mathscr{W})$ does not depend on the choice of $\chi$. In particular $\Psi(\mathscr{W})$ is Verdier self-dual, since $\Psi$ commutes with Verdier duality and the Verdier dual of $\mathscr{W}$ is the clean extension of $\ev^{\Delta}\chi^{-1}$.

\label{intro2}

\subsection{} The stack $\overline{\Bun}_{N^{\omega}}$ has a stratification by defect of the generalized $N$-bundle. Our first main theorem is a description of the restrictions of $\Psi(\mathscr{W})$ to the strata. In particular we will prove that this sheaf is tilting with respect to the defect stratification, meaning its $!$- and $*$-restrictions to the strata are perverse (although generally not lisse).

Write $\Lambda$ for the lattice of cocharacters of $T$, and denote by $\Lambda^{\pos} \subset \Lambda$ the positive coweights with respect to $B$. The defect stratification of $\overline{\Bun}_{N^{\omega}}$ is indexed by $\Lambda^{\pos}$, and for each $\mu \in \Lambda^{\pos}$ we denote the locally closed stratum embedding by \[ \mathfrak{j}_{=\mu} : \overline{\Bun}_{N^{\omega},=\mu} \longrightarrow \overline{\Bun}_{N^{\omega}}. \] If $\mu = \sum_{i \in I} n_i\alpha_i$ for some nonnegative integers $n_i$ (here $\alpha_i$ is the simple coroot corresponding to $i$) then the corresponding configuration space of points in $X$ is defined by \[ X^{\mu} := \prod_{i \in I} X^{(n_i)}, \] and there is a smooth surjection \[ \mathfrak{m}_{\mu} : \overline{\Bun}_{N^{\omega},=\mu} \longrightarrow X^{\mu}. \]

In \cite{BG2} the authors introduced certain factorizable perverse sheaves $\Omega^{\mu}$ on the configuration spaces $X^{\mu}$. They can be characterized as follows: the $!$-fiber of $\Omega^{\mu}$ at a point $\sum_i \mu_ix_i$ is \[ \bigotimes_i C^{\bullet}(\check{\mathfrak{n}})^{\mu_i}, \] where $\check{\mathfrak{n}}$ is the unipotent radical of a Borel subalgebra in the Langlands dual Lie algebra $\check{\mathfrak{g}}$, and $C^{\bullet}(\check{\mathfrak{n}})$ is its $\check{T}$-graded cohomological Chevalley complex.

\begin{theorem}

For any $\mu \in \Lambda^{\pos}$ there is an isomorphism $\mathfrak{j}_{=\mu}^!\Psi(\mathscr{W}) \tilde{\to} \mathfrak{m}_{\mu}^{\Delta}\Omega^{\mu}$.

\label{strata}

\end{theorem}

In fact, $\Omega^{\mu}$ is indecomposable and $\mathfrak{m}_{\mu}$ has contractible fibers, so the isomorphism in the theorem is automatically unique up to scaling. Since $\Psi(\mathscr{W})$ is Verdier self-dual, Theorem \ref{strata} implies that it is tilting with respect to the defect stratification.

\subsection{} Recall that the \emph{monodromy filtration} on $\Psi(\mathscr{W})$ is the unique filtration by perverse sheaves \[ F_{-m} \subset \cdots \subset F_{m-1} \subset F_m = \Psi(\mathscr{W}) \] such that for all $1 \leq i \leq n$, the $i^{\text{th}}$ power of the monodromy endomorphism induces an isomorphism $F_i/F_{i-1} \tilde{\to} F_{-i}/F_{-i-1}$. The associated graded sheaf $\gr \Psi(\mathscr{W})$ has an action of the so-called Lefschetz $\mathfrak{sl}_2$ such that the lowering operator is induced by the monodromy endomorphism and $F_i/F_{i-1}$ has weight $i$ for the Cartan operator.

We now formulate a description of $\gr \Psi(\mathscr{W})$ in terms of certain factorizable perverse sheaves with $\mathfrak{sl}_2$-action, called the \emph{Picard-Lefschetz oscillators} after \cite{S1} and \cite{S2} (see also Section  3.2 of \emph{loc. cit.}).

First let us define the Picard-Lefschetz oscillators on $X^{(n)}$. Let $\std$ denote the standard $2$-dimensional representation of the Lefschetz $\mathfrak{sl}_2$, and write $\sgn$ for the sign character of the symmetric group $\Sigma_n$. Now the $\Sigma_n \times \mathfrak{sl}_2$-representation $\sgn \otimes \std^{\otimes n}$ (here $\Sigma_n$ also permutes the $\std$ factors) and the $\Sigma_n$-torsor $X^n_{\disj} \to X^{(n)}_{\disj}$ give rise to a local system on $X^{(n)}_{\disj}$ with $\mathfrak{sl}_2$-action. Then $\mathscr{P}_n$ is defined as the intermediate (i.e. Goresky-MacPherson) extension to $X^{(n)}$ of the perverse cohomological shift of this local system. The perverse sheaf $\mathscr{P}_n$ carries an $\mathfrak{sl}_2$-action by functoriality, and is evidently semisimple.

Recall that a \emph{Kostant partition} of $\mu \in \Lambda^{\pos}$ is an expression of the form $\mu = \sum_{\beta \in R^+} n_{\beta}\beta$, where the $n_{\beta}$ are nonnegative integers and $R^+$ denotes the set of positive coroots. Denote by $\Kost(\mu)$ the set of Kostant partitions of $\mu$. To any $\mathfrak{k} \in \Kost(\mu)$ given by $\mu = \sum_{\beta \in R^+} n_{\beta}\beta$ we attach a partially symmetrized power \[ X^{\mathfrak{k}} := \prod_{\beta \in R^+} X^{(n_{\beta})}, \] which is equipped with a canonical finite map $\iota^{\mathfrak{k}} : X^{\mathfrak{k}} \to X^{\mu}$.

The \emph{Picard-Lefschetz oscillator} on $X^{\mu}$ is defined by the formula \[ \mathscr{P}^{\mu} := \bigoplus_{\mathfrak{k} \in \Kost(\mu)} \iota^{\mathfrak{k}}_*(\underset{\beta \in R^+}{\boxtimes} \mathscr{P}_{n_{\beta}}). \] In particular, we have $\mathscr{P}^{n\alpha} = \mathscr{P}_n$ for $\alpha$ a simple coroot. By construction, $\mathscr{P}^{\mu}$ is a semisimple perverse sheaf with $\mathfrak{sl}_2$-action.

\begin{theorem}

There is an $\mathfrak{sl}_2$-equivariant isomorphism
\begin{equation}
\gr \Psi(\mathscr{W}) \tilde{\longrightarrow} \bigoplus_{\mu \in \Lambda^{\pos}} \mathfrak{j}_{=\mu,!*}\mathfrak{m}_{\mu}^{\Delta}\mathscr{P}^{\mu}.
\label{maineq}
\end{equation}

\label{main}

\end{theorem}

In particular $\gr \Psi(\mathscr{W})$ is semisimple. Using the theorem we can compute the kernel of the monodromy operator, whose simple subquotients are the lowest weight sheaves for the Lefschetz $\mathfrak{sl}_2$-action.

\begin{corollary}

The canonical morphism $\mathfrak{j}_{=0,!}\IC_{\Bun_{N^{\omega}}} \to \Psi(\mathscr{W})$ is an isomorphism onto the kernel of the monodromy operator.

\label{kermon}

\end{corollary}

\subsection{} A key step in the proof of Theorem \ref{main} is establishing that $\gr \Psi(\mathscr{W})$ is semisimple. This can be shown using Mochizuki's theory of weights for holonomic D-modules, developed in \cite{M}. We also present a proof of Theorem \ref{main} which does not use weights for irregular D-modules, but instead relies Proposition 4.9 in \cite{FFKM}. We now state a conjecture containing this proposition as a special case, which we hope to return to in future work. In particular we will explain how this conjecture can be deduced from the (conjectural) geometric Langlands equivalence and its postulated compatibility with compactified Eisenstein series.

Write $\overline{\mathfrak{p}} : \overline{\Bun}_B \to \Bun_G$ and $\overline{\mathfrak{q}} : \overline{\Bun}_B \to \Bun_T$ for the canonical morphisms. The functor of \emph{compactified Eisenstein series} $\Eis_{!*} : D(\Bun_T) \to D(\Bun_G)$, introduced in \cite{BG1}, is defined by \[ \Eis_{!*} = \overline{\mathfrak{p}}_*(\IC_{\overline{\Bun}_B} \otimes^! \, \overline{\mathfrak{q}}^!(-)). \]

Recall that geometric class field theory assigns to a $\check{T}$-local system $E_{\check{T}}$ on $X$ a character sheaf (i.e. multiplicative line bundle with connection) $\mathscr{L}(E_{\check{T}})$ on $\Bun_T$. We write $\check{\mathfrak{g}}_{E_{\check{T}}}$ for the local system on $X$ attached to $E_{\check{T}}$ and the $\check{T}$-representation $\check{\mathfrak{g}}$.

\begin{conjecture}

The DG Lie algebra $H^{\bullet}(X,\check{\mathfrak{g}}_{E_{\check{T}}})$ acts on $\Eis_{!*}\mathscr{L}(E_{\check{T}})$.

\label{dualact}

\end{conjecture}

For the trivial $\check{T}$-local system, which is the only case we will use, the conjecture says that $\check{\mathfrak{g}} \otimes H^{\bullet}(X)$ acts on $\overline{\mathfrak{p}}_*\IC_{\overline{\Bun}_B}$. In Section \ref{mainproof2} we specify the action of certain generators of $\check{\mathfrak{g}} \otimes H^{\bullet}(X)$. In \cite{FFKM} the authors verify the necessary relations for $\check{\mathfrak{g}} \otimes H^0(X)$, which suffices for our application. The derivedness of the Lie algebra $\check{\mathfrak{g}} \otimes H^{\bullet}(X)$, or more generally $H^{\bullet}(X,\check{\mathfrak{g}}_{E_{\check{T}}})$, makes this checking of relations difficult to do ``by hand," and in any case a more conceptual approach to the conjecture is desirable.

\subsection*{Acknowledgments:} I thank my doctoral advisor Dennis Gaitsgory for suggesting that I study this particular nearby cycles sheaf, and for many invaluable conversations full of technical assistance and motivation. The ways in which he has helped and guided me through this project are too numerous to list, but I am especially grateful to him for explaining the statement of Conjecture \ref{dualact} and its relevance to the semisimplicity assertion of Theorem \ref{main}. I thank Sam Raskin for many illuminating discussions about this material, which included helping me to formulate Lemma \ref{nclem} and explaining Definition \ref{ncconstr}. I am grateful to Simon Schieder for explaining to me the definition of the Picard-Lefschetz oscillators for an arbitrary reductive group.

\section{Drinfeld compactifications and Zastava spaces}

\subsection{} Let $2\rho : \mathbb{G}_m \to T$ denote the sum of the simple coroots and fix a square root $\omega_X^{\otimes \frac12}$ of $\omega_X$. We define \[ \rho(\omega_X) := 2\rho(\omega_X^{\otimes \frac12}) \in \Bun_T(k). \]

By definition, $\overline{\Bun}_{N^{\omega}}$ is the fiber product \[
\begin{tikzcd}[cramped]
\overline{\Bun}_{N^{\omega}} \ar[d] \ar[r] & \overline{\Bun}_B \ar[d] \\
\Spec k \ar[r, "\rho(\omega_X)"] & \Bun_T.
\end{tikzcd} \]
The $T$-bundle $\rho(\omega_X)$ and the action of $T$ on $N$ give rise to a group scheme $N^{\omega}$ over $X$, and the open stratum $\Bun_{N^{\omega}} \to \overline{\Bun}_{N^{\omega}}$ is identified with the moduli stack of $N^{\omega}$-bundles, as the notation suggests. By construction, $T$ acts on $\overline{\Bun}_{N^{\omega}}$ in such a way that $\overline{\Bun}_{N^{\omega}} \to \overline{\Bun}_B$ factors through a closed embedding \[ \overline{\Bun}_{N^{\omega}}/T \longrightarrow \overline{\Bun}_B. \]

For each $\mu \in \Lambda^{\pos}$ the corresponding stratum $\overline{\Bun}_{N^{\omega},=\mu}$ fits into a fiber square \[
\begin{tikzcd}[cramped]
\overline{\Bun}_{N^{\omega},=\mu} \ar[d] \ar[r] & \Bun_B \ar[d] \\
X^{\mu} \ar[r] & \Bun_T,
\end{tikzcd} \] where the lower horizontal morphism is the twisted Abel-Jacobi map $D \mapsto \rho(\omega_X)(D)$. We write \[ \overline{\Bun}_{N^{\omega},\leq \mu} := \bigcup_{\mu' \leq \mu} \overline{\Bun}_{N^{\omega},=\mu'} \] and $\mathfrak{j}_{\leq \mu}$ for the corresponding open embedding.

The embedding $\mathfrak{j}_{=\mu}$ of the stratum, which is known to be affine, extends to a finite map \[ \mathfrak{j}_{\geq \mu} : \overline{\Bun}_{N^{\omega},\geq \mu} \longrightarrow \overline{\Bun}_{N^{\omega}}, \] where $\overline{\Bun}_{N^{\omega},\geq \mu} := X^{\mu} \times_{\Bun_T} \overline{\Bun}_B$.

\subsection{} Now we introduce the Zastava spaces, which are factorizable local models for $\overline{\Bun}_{N^{\omega}}$. There are several versions of Zastava space, and notations vary significantly within the literature (ours is similar to \cite{AG}).

Define the Zastava space $Z$ to be the open locus in $\overline{\Bun}_{N^{\omega}} \times_{\Bun_G} \Bun_{B^-}$ where the generalized $N$-reduction and $B^-$-reduction are transverse generically on $X$. It is well-known that $Z$ is a scheme, with connected components \[ Z^{\lambda} := Z \cap (\overline{\Bun}_{N^{\omega}} \times_{\Bun_G} \Bun_{B^-}^{\lambda + \deg \rho(\omega_X)}) \] indexed by $\lambda \in \Lambda^{\pos}$.

The map $\mathfrak{p}^- : \Bun_{B^-} \to \Bun_G$ gives rise to $^{\prime}\mathfrak{p}^- : Z \to \overline{\Bun}_{N^{\omega}}$. It is shown in \cite{BG1} that $\mathfrak{p}^{-,\lambda}$ and therefore $^{\prime}\mathfrak{p}^{-,\lambda}$ are smooth for $\lambda$ sufficiently dominant. Moreover, given a quasicompact open $U \subset \overline{\Bun}_{N^{\omega}}$, for $\lambda$ sufficiently dominant the image of $Z^{\lambda} \to \overline{\Bun}_{N^{\omega}}$ contains $U$.

For each $0 \leq \mu \leq \lambda$ we have the corresponding stratum $Z^{\lambda}_{=\mu} := Z^{\lambda} \times_{\overline{\Bun}_{N^{\omega}}} \overline{\Bun}_{N^{\omega},=\mu}$ with locally closed embedding \[ ^{\prime}\mathfrak{j}^{\lambda}_{=\mu} : Z^{\lambda}_{=\mu} \longrightarrow Z^{\lambda}. \] Define $\mathring{Z}^{\lambda} := Z^{\lambda}_{=0}$. Similarly, we have $Z^{\lambda}_{\leq\mu} := Z^{\lambda} \times_{\overline{\Bun}_{N^{\omega}}} \overline{\Bun}_{N^{\omega},\leq \mu}$ with the open embedding \[ ^{\prime}\mathfrak{j}^{\lambda}_{\leq\mu} : Z^{\lambda}_{\leq\mu} \longrightarrow Z^{\lambda}. \] Put $Z^{\lambda}_{\geq \mu} := Z \times_{\overline{\Bun}_{N^{\omega}}} \overline{\Bun}_{N^{\omega},\geq \mu}$, so that $^{\prime}\mathfrak{j}^{\lambda}_{=\mu}$ extends to the finite map \[ ^{\prime}\mathfrak{j}^{\lambda}_{\geq \mu} : Z^{\lambda}_{\geq \mu} \longrightarrow Z^{\lambda}. \]

\subsection{} Let us recall the factorization structure on $Z$. For every $\lambda \in \Lambda^{\pos}$ there is a canonical map $\pi^{\lambda} : Z^{\lambda} \to X^{\lambda}$, which is well-known to be affine. We write $(X^{\lambda_1} \times X^{\lambda_2})_{\disj} \subset X^{\lambda_1} \times X^{\lambda_2}$ for the open locus where the two divisors are disjoint. Similarly, put \[ (Z^{\lambda_1} \times Z^{\lambda_2})_{\disj} := (Z^{\lambda_1} \times Z^{\lambda_2}) \times_{X^{\lambda_1} \times X^{\lambda_2}} (X^{\lambda_1} \times X^{\lambda_2})_{\disj}. \] The factorization structure is a canonical morphism $(Z^{\lambda_1} \times Z^{\lambda_2})_{\disj} \to Z^{\lambda_1 + \lambda_2}$ which fits into a fiber square \[
\begin{tikzcd}[cramped]
(Z^{\lambda_1} \times Z^{\lambda_2})_{\disj} \ar[d] \ar[r] & Z^{\lambda_1 + \lambda_2} \ar[d] \\
(X^{\lambda_1} \times X^{\lambda_2})_{\disj} \ar[r] & X^{\lambda_1 + \lambda_2}.
\end{tikzcd} \]

The factorization structure is compatible with the defect stratification in the following sense. The factorization structure on the strata consists of, for each decomposition $\mu_1 + \mu_2 = \mu$ satisfying $0 \leq \mu_1 \leq \lambda_1$ and $0 \leq \mu_2 \leq \lambda_2$, a morphism $(Z^{\lambda_1}_{=\mu_1} \times Z^{\lambda_2}_{=\mu_2})_{\disj} \to Z^{\lambda_1 + \lambda_2}_{=\mu}$ which fits into a fiber square \[
\begin{tikzcd}[cramped]
\coprod_{\mu_1 + \mu_2 = \mu} (Z^{\lambda_1}_{=\mu_1} \times Z^{\lambda_2}_{=\mu_2})_{\disj} \ar[d] \ar[r] & Z^{\lambda_1 + \lambda_2}_{=\mu} \ar[d] \\
(X^{\lambda_1} \times X^{\lambda_2})_{\disj} \ar[r] & X^{\lambda_1 + \lambda_2}.
\end{tikzcd} \] One has similar factorization structures on $Z^{\lambda}_{\leq \mu}$ and $Z^{\lambda}_{\geq \mu}$. Moreover, these factorization structures are compatible with $^{\prime}\mathfrak{j}^{\lambda}_{=\mu}$, $^{\prime}\mathfrak{j}^{\lambda}_{\geq \mu}$, etc.

\label{factsec}

\subsection{} We will also need the compactified Zastava space $\overline{Z}$, which is the open locus in $\overline{\Bun}_{N^{\omega}} \times_{\Bun_G} \overline{\Bun}_{B^-}$ where the generalized $N$- and $B^-$-reductions are generically transverse. In particular there is an open embedding $^{\prime}\mathfrak{j}^- : Z \to \overline{Z}$ obtained from $\mathfrak{j}^- : \Bun_{B^-} \to \overline{\Bun}_{B^-}$ by base change.

For any $\nu \in \Lambda^{\pos}$ we put \[ _{=\nu}\overline{Z} := \overline{Z} \times_{\overline{\Bun}_{B^-}} \overline{\Bun}_{B^-,=\nu}, \] and similarly for $_{\leq \nu}\overline{Z}$ and $_{\geq \nu}\overline{Z}$.

The projections $\pi^{\lambda}$ extend to proper morphisms $\overline{\pi}^{\lambda} : \overline{Z}^{\lambda} \to X^{\lambda}$. The factorization structure on $Z$ extends to $\overline{Z}$ in a way compatible with both defect stratifications on $\overline{Z}$.

\section{Nearby cycles and adelic invariance}

\subsection{} We will need a slightly nonstandard construction of nearby cycles for the first proof of Theorem \ref{strata}. Let $Y$ be a scheme of finite type equipped with an action of $\G_m$ and a $\G_m$-equivariant morphism $f : Y \to \A^1$. Write $Y_0$ for the fiber of $Y$ over $0$ and $\mathring{Y}$ for the preimage of $\G_m$. We temporarily denote the embeddings by \[ Y_0 \stackrel{i}{\longrightarrow} Y \stackrel{j}{\longleftarrow} \mathring{Y}. \]

Observe that for any holonomic D-module $\mathscr{F}$ on $\mathring{Y}$, we have an action of $H^{\bullet}(\G_m)$ on $\mathscr{F}$ and hence, by functoriality, on $i^!j_!\mathscr{F}$. The point $1 \in \G_m$ induces an augmentation $H^{\bullet}(\G_m) \to k$, and since $H^{\bullet}(\G_m)$ is generated by a single element in cohomological degree $1$, the endomorphism algebra $\End_{H^{\bullet}(\G_m)}(k)$ is canonically isomorphic to the polynomial ring $k[t]$.

\begin{definition}

We define unipotent nearby cycles with respect to $f$ to be the functor \[ \Psi : D_{\hol}(\mathring{Y}) \longrightarrow D(Y_0) \] given by the formula \[ \Psi(\mathscr{F}) = k \otimes_{H^{\bullet}(\G_m)} i^!j_!\mathscr{F}. \] The action of $\End_{H^{\bullet}(\G_m)}(k) = k[t]$ defines the monodromy endomorphism of $\Psi$.

\label{ncconstr}

\end{definition}

\begin{proposition}

The functor $\Psi$ has the following properties:

\begin{enumerate}[(i)]
\item it coincides with the construction in \cite{B} and in particular preserves holonomicity,
\item if $\mathscr{F}$ is $\G_m$-equivariant then $\Psi(\mathscr{F})$ is unipotently $\G_m$-monodromic, and the monodromy endomorphism is the obstruction to $\G_m$-equivariance.
\end{enumerate}

\label{ncprop}

\end{proposition}

\begin{proof}

Let $\mathscr{F}$ be a holonomic D-module on $\mathring{Y}$. We recall Beilinson's construction: for any $a \geq 1$ let $L_a$ be the shifted D-module on $\G_m$ corresponding to the local system whose monodromy is a unipotent Jordan block of rank $a$. There is a canonical map $L_a \to L_{a+1}$. Then Beilinson's definition of nearby cycles is \[ \colim_a i^!j_!(\mathscr{F} \stackrel{*}{\otimes} f^*L_a). \] Moreover, this colimit is isomorphic to $H^0i^!j_!(\mathscr{F} \stackrel{*}{\otimes} f^*L_a)$ for large $a$, which is evidently holonomic.

Observe that we can resolve the augmentation module for $H^{\bullet}(\G_m)$ using the total complex of the double complex \[ \cdots \longrightarrow H^{\bullet}(\G_m)[-2] \longrightarrow H^{\bullet}(\G_m)[-1] \longrightarrow H^{\bullet}(\G_m). \] Tensoring this with $i^!j_!\mathscr{F}$, we obtain a double complex
\begin{equation}
\cdots \longrightarrow i^!j_!\mathscr{F}[-2] \longrightarrow i^!j_!\mathscr{F}[-1] \longrightarrow i^!j_!\mathscr{F}
\label{doubcom}
\end{equation}
whose total complex is quasi-isomorphic to $\Psi(\mathscr{F})$. Note that the transition maps $i^!j_!\mathscr{F}[-n] \to i^!j_!\mathscr{F}[-n+1]$ are given by the action of a generator of $H^1(\G_m)$. Thus the truncated double complex \[ i^!j_!\mathscr{F}[-a+1] \longrightarrow \cdots \longrightarrow i^!j_!\mathscr{F}[-1] \longrightarrow i^!j_!\mathscr{F} \] has total complex quasi-isomorphic to $i^!j_!(\mathscr{F} \stackrel{*}{\otimes} f^*L_a)$, which follows from the observation that the cofiber of the composition \[ k_{\mathbb{G}_m}[-1] \longrightarrow k_{\mathbb{G}_m} \longrightarrow L_a \] is $L_{a+1}$. Here the first map is the action of $H^1(\mathbb{G}_m)$, and the second is the canonical inclusion. Since the total complex of (\ref{doubcom}) is the colimit of its truncations, the assertion (i) follows.

For (ii), note that the $\G_m$-equivariance of $i^!j_!\mathscr{F}$ implies that $\Psi(\mathscr{F}) = H^0i^!j_!(\mathscr{F} \stackrel{*}{\otimes} f^*L_a)$ is $\G_m$-monodromic (here $a$ is large). By construction the monodromy endomorphism is induced by the canonical endomorphism of $L_a$ with one-dimensional kernel and cokernel. But the latter is precisely the obstruction to $\G_m$-equivariance for $L_a$, so the claim follows from the functoriality of this obstruction.

\end{proof}

It follows from part (i) of Proposition \ref{ncprop} that $\Psi$ enjoys the standard properties of the unipotent nearby cycles functor: it is t-exact, commutes with Verdier duality, and commutes with proper direct image and smooth inverse image.

\label{ncsec}

\subsection{} Before proving Theorem \ref{strata}, we will show that $\mathfrak{j}_{=\mu}^!\Psi(\mathscr{W})$ is pulled back from $X^{\mu}$ for any $\mu \in \Lambda^{\pos}$. This property is equivalent to invariance under the ``adelic $N^{\omega}$," as we now explain.

For any $x \in X$, we define the open substack $\overline{\Bun}_{N^{\omega}}^x \subset \overline{\Bun}_{N^{\omega}}$ to consist of those generalized $N^{\omega}$-bundles whose defect is disjoint from $x$. A point of the ind-algebraic stack $\mathscr{H}_{N^{\omega}}^x$ consists of two points of $\overline{\Bun}_{N^{\omega}}^x$ together with an identification over $X \setminus \{ x \}$. Note that $\mathscr{H}_{N^{\omega}}^x$ has the structure of a groupoid acting on $\overline{\Bun}_{N^{\omega}}^x$. The fibers of $\mathscr{H}_{N^{\omega}}^x$ over $\overline{\Bun}_{N^{\omega}}^x \times \overline{\Bun}_{N^{\omega}}^x$ are isomorphic to ind-affine space $\colim_n \A^n$, which implies that the functor which forgets $\mathscr{H}_{N^{\omega}}^x$-equivariance is fully faithful, i.e. $\mathscr{H}_{N^{\omega}}^x$-equivariance is a property.

We say that an object of $D(\overline{\Bun}_{N^{\omega}})$ is $N^{\omega}(\A)$\emph{-equivariant} if, for every $x \in X$, its restriction to $\overline{\Bun}_{N^{\omega}}^x$ is $\mathscr{H}_{N^{\omega}}^x$-equivariant.

\begin{proposition}

An object $\mathscr{F}$ of $D(\overline{\Bun}_{N^{\omega}})$ is $N^{\omega}(\A)$-equivariant if and only if, for every $\mu \in \Lambda^{\pos}$, the canonical morphism \[ \mathfrak{m}_{\mu}^*\mathfrak{m}_{\mu,*}\mathfrak{j}^!_{=\mu}\mathscr{F} \longrightarrow \mathfrak{j}^!_{=\mu}\mathscr{F} \] is an isomorphism.

\label{adelicstrata}

\end{proposition}

For each $x \in X$, denote by $\mathscr{O}_x$ the completed local ring of $X$ at $x$, with fraction field $K_x$. If $R$ is a $k$-algebra, we denote by $R \hat{\otimes} \mathscr{O}_x$ and $R \hat{\otimes} K_x$ the respective completed tensor products.

The local Hecke stack $\mathscr{H}^{\loc,x}_{N^{\omega}}$ is defined as follows: a $\Spec R$-point of $\mathscr{H}^{\loc,x}_{N^{\omega}}$ consists of two $N^{\omega}$-bundles over $\Spec(R \hat{\otimes} \mathscr{O}_x)$ equipped with an isomorphism over $\Spec(R \hat{\otimes} K_x)$. There is a natural restriction map $\res^x : \mathscr{H}^x_{N^{\omega}} \to \mathscr{H}^{\loc,x}_{N^{\omega}}$. Moreover, our choice of $\psi$ induces a map $\mathscr{H}^{\loc,x}_{N^{\omega}} \to \G_a$ in the following way. Using $\psi$ we obtain an isomorphism $[N,N] \cong \G_a^{\oplus I}$, so the projection $N \to [N,N]$ induces a morphism \[ \mathscr{H}^{\loc,x}_{N^{\omega}} \longrightarrow \prod_I \mathscr{H}^{\loc,x}_{\G_a^{\omega}}, \] where $\G_a^{\omega} := \G_a \times^{\G_m} \omega_X$. Note that there is a natural isomorphism
\begin{align*}
\mathscr{H}^{\loc,x}_{\G_a^{\omega}} &\tilde{\longrightarrow} \Gamma(\Spec K_x,\omega_X)/\Gamma(\Spec \mathscr{O}_x,\omega_X) \\
&\tilde{\longrightarrow} \Gamma(X \setminus \{ x \},\omega_X)/\Gamma(X,\omega_X),
\end{align*}
which defines a canonical morphism $\rsd^x : \mathscr{H}^{\loc,x}_{\G_a^{\omega}} \to \G_a$. The composition \[ \rsd^x_{\psi} : \mathscr{H}^x_{N^{\omega}} \stackrel{\res^x}{\longrightarrow} \mathscr{H}^{\loc,x}_{N^{\omega}} \longrightarrow \prod_I \mathscr{H}^{\loc,x}_{\G_a^{\omega}} \stackrel{\prod \rsd^x}{\longrightarrow} \prod_I \G_a \stackrel{\add}{\longrightarrow} \G_a \] is an additive character, meaning it is a morphism of groupoids. Thus $\widetilde{\chi}^x := \rsd^{x,!}_{\psi}\chi$ is a character sheaf on $\mathscr{H}_{N^{\omega}}^x$, and we can speak of $(\mathscr{H}_{N^{\omega}}^x,\widetilde{\chi}^x)$-equivariant sheaves on $\overline{\Bun}_{N^{\omega}}^x$, which form a full subcategory of $D(\overline{\Bun}_{N^{\omega}}^x)$. Likewise, if a sheaf on $\overline{\Bun}_{N^{\omega}}$ is $(\mathscr{H}_{N^{\omega}}^x,\widetilde{\chi}^x)$-equivariant for all $x \in X$ we say that it is $(N^{\omega}(\A),\widetilde{\chi})$\emph{-equivariant}. Although we will not use this fact, the category of $(N^{\omega}(\A),\widetilde{\chi})$-equivariant sheaves on $\overline{\Bun}_{N^{\omega}}$ is equivalent to the category of vector spaces, being generated by $\mathscr{W}$. Moreover $(\mathscr{H}_{N^{\omega}}^x,\widetilde{\chi}^x)$-equivariance for a single $x \in X$ implies $(N^{\omega}(\A),\widetilde{\chi})$-equivariance.

Observe that there is a natural $T$-action on $\mathscr{H}^{\loc,x}_{N^{\omega}}$. Using the chosen dominant regular cocharacter $\gamma : \G_m \to T$, the resulting $\G_m$-action contracts $\mathscr{H}^{\loc,x}_{N^{\omega}}$ to a point. In particular, it extends to an action $\A^1 \times \mathscr{H}^{\loc,x}_{N^{\omega}} \to \mathscr{H}^{\loc,x}_{N^{\omega}}$ of the multiplicative monoid $\A^1$. The $!$-pullback of $\chi$ along the composition \[ \A^1 \times \mathscr{H}^x_{N^{\omega}} \stackrel{\id_{\A^1} \times \res^x}{\longrightarrow} \A^1 \times \mathscr{H}^{\loc,x}_{N^{\omega}} \longrightarrow \mathscr{H}^{\loc,x}_{N^{\omega}} \longrightarrow \G_a \] defines an $\A^1$-family $\widetilde{\chi}^x_{\ext}$ of character sheaves on $\mathscr{H}^x_{N^{\omega}}$. Its $!$-restriction to $\{ 1 \} \times \mathscr{H}^x_{N^{\omega}}$ is $\widetilde{\chi}_x$, and it is trivial along $\{ 0 \} \times \mathscr{H}^x_{N^{\omega}}$.

\begin{lemma}

The sheaf $\Psi(\mathscr{W})$ is $N^{\omega}(\A)$-equivariant.

\label{adelicinv}

\end{lemma}

\begin{proof}

Fix $x \in X$; we omit restriction to $\overline{\Bun}_{N^{\omega}}^x$ from the notation in what follows. By construction $\mathscr{W}$ is $\widetilde{\chi}^x_{\ext}|_{\G_m \times \mathscr{H}^x_{N^{\omega}}}$-equivariant. Since \[ \widetilde{\chi}^x_{\ext}|^!_{\{ 0 \} \times \mathscr{H}^x_{N^{\omega}}} = \omega_{\mathscr{H}^x_{N^{\omega}}} \] is the trivial character sheaf, it follows from Proposition \ref{ncconstr} that $\Psi(\mathscr{W})$ is $\mathscr{H}^x_{N^{\omega}}$-equivariant as desired.

\end{proof}

\section{Restriction to the strata}

\label{stratasec}

\subsection{} Now we give the first proof of Theorem \ref{strata} by deducing it from Theorem 1.3.6 in \cite{AG}, which describes the restrictions to the strata of $i!j_!\mathscr{W}$ in terms of the perverse sheaf $\Omega$. Since we work with a fixed dominant regular coweight $\gamma$ rather than the entire torus $T$, it will be necessary to prove a slightly different formulation of the latter theorem.

The inclusion of $N^{\omega}(\A)$-equivariant sheaves on $\overline{\Bun}_{N^{\omega}}$ admits a right adjoint, which we denote by $\Av_*^{N^{\omega}(\A)}$. Let $i$ and $j$ be as in Section \ref{intro2}, and write $p : \A^1 \times \overline{\Bun}_{N^{\omega}} \to \overline{\Bun}_{N^{\omega}}$ for the projection. We will also abusively denote $p \circ j$ by $p$.

\begin{proposition}

There is a canonical isomorphism $i^!j_!\mathscr{W} \tilde{\to} \Av^{N^{\omega}(\A)}_*p_!\mathscr{W}$.

\label{kir}

\end{proposition}

\begin{proof}

See Section 10.3 of \cite{AG}, where the claim is proved for the action of the entire torus $T$. The same proof applies \emph{mutatis mutandi} to our claim, which involves only the $\G_m$-action.

\end{proof}

Recall that by Proposition \ref{ncconstr}, we have an isomorphism \[ \Psi(\mathscr{W}) \tilde{\longrightarrow} k \otimes_{H^{\bullet}(\G_m)} i^!j_!\mathscr{W}. \]

\begin{proof}[First proof of Theorem \ref{strata}]

Applying Proposition \ref{ncconstr} and the fact that $\mathfrak{j}_{=\mu}^!$ preserves (homotopy) colimits, we see that \[ \mathfrak{j}_{=\mu}^!\Psi(\mathscr{W}) \tilde{\longrightarrow} k \otimes_{H^{\bullet}(\G_m)} \mathfrak{j}_{=\mu}^!i^!j_!\mathscr{W}. \] By Proposition \ref{kir} we have \[ \mathfrak{j}_{=\mu}^!i^!j_!\mathscr{W} \tilde{\longrightarrow} \mathfrak{j}^!_{=\mu}\Av^{N^{\omega}(\A)}_*p_!\mathscr{W}, \] and Proposition \ref{adelicstrata} implies that \[ \mathfrak{j}^!_{=\mu}\Av^{N^{\omega}(\A)}_*p_!\mathscr{W} \tilde{\longrightarrow} \mathfrak{m}_{\mu}^*\mathfrak{m}_{\mu,*}\mathfrak{j}^!_{=\mu}p_!\mathscr{W}. \]

Now Theorem 1.3.6 of \cite{AG} yields
\begin{equation}
\mathfrak{m}_{\mu}^*\mathfrak{m}_{\mu,*}\mathfrak{j}^!_{=\mu}p_!\mathscr{W} \tilde{\longrightarrow} \mathfrak{m}_{\mu}^{\Delta}\Omega^{\mu} \otimes H^{\bullet}_c(\G_m)[1].
\label{braden}
\end{equation}
It remains to show that under the composed isomorphism \[ \mathfrak{j}_{=\mu}^!i^!j_!\mathscr{W} \tilde{\longrightarrow} \mathfrak{m}_{\mu}^{\Delta}\Omega^{\mu} \otimes H^{\bullet}_c(\G_m)[1], \] the action of $H^{\bullet}(\G_m)$ on the left hand side corresponds to the natural action on $H^{\bullet}_c(\G_m)$ on the right hand side. Since $H^{\bullet}_c(\G_m) = H^{\bullet}(\G_m)[-1]$ as $H^{\bullet}(\G_m)$-modules, this will finish the proof.

It is clear that \[ \mathfrak{j}_{=\mu}^!i^!j_!\mathscr{W} \tilde{\longrightarrow} \mathfrak{m}_{\mu}^*\mathfrak{m}_{\mu,*}\mathfrak{j}^!_{=\mu}p_!\mathscr{W} \] intertwines the actions of $H^{\bullet}(\G_m)$, since it is obtained by evaluating a morphism of functors on $\mathscr{W}$. Tracing through the proof of Theorem 1.3.6 in \cite{AG}, we see that the isomorphism (\ref{braden}) is also obtained by evaluating a morphism of functors on $\mathscr{W}$, with the appearance of $H^{\bullet}_c(\G_m)$ accounted for by the isomorphism \[ a_{\gamma,!}\mathscr{W} \tilde{\longrightarrow} \mathscr{W}_1 \otimes H_{\bullet}(\G_m)[-1] \tilde{\longrightarrow} \mathscr{W}_1 \otimes H_c^{\bullet}(\G_m)[1]. \] The latter isomorphism intertwines the actions of $H^{\bullet}(\G_m)$ as needed.

\end{proof}

\subsection{} The rest of this subsection is devoted to the second proof of Theorem \ref{strata}. This proof applies Theorem 4.6.1 in \cite{R}, which says that $\Omega$ can be realized as the twisted cohomology of Zastava space. Accordingly, we must formulate the analogue of Theorem \ref{strata} on Zastava space. First, the Whittaker sheaf: we claim that \[ \mathscr{W}_{Z^{\lambda}} := (\id_{\G_m} \times \ \! ^{\prime}\mathfrak{p}^{-,\lambda})^!\mathscr{W}[\dim \overline{\Bun}_{N^{\omega}} - \dim Z^{\lambda}] \] is perverse for any $\lambda$. For $\lambda$ sufficiently dominant $\mathfrak{p}^{-,\lambda}$ is smooth, so that $\mathscr{W}_{Z^{\lambda}}$ is the cohomologically normalized inverse image of the perverse sheaf $\mathscr{W}$. If $\lambda' \leq \lambda$ then we can pull back $\mathscr{W}_{Z^{\lambda}}$ along $\id_{\G_m}$ times the factorization morphism \[ (Z^{\lambda'} \times Z^{\lambda-\lambda'})_{\disj} \longrightarrow Z^{\lambda}, \] and it is not hard to see that we obtain the restriction of $\mathscr{W}_{Z^{\lambda'}} \boxtimes \mathscr{W}_{Z^{\lambda-\lambda'}}$. Since the factorization map is \'{e}tale, this implies that $\mathscr{W}_{Z^{\lambda'}}$ is perverse as desired.

Since the map $\overline{\Bun}_{N^{\omega}} \to \Bun_G$ is $T$-equivariant for the trivial action of $T$ on $\Bun_G$, we obtain an action of $T$ on $Z$ which makes $^{\prime}\mathfrak{p}^-$ a $T$-equivariant map. As with $\mathscr{W}$, we also denote by $\mathscr{W}_{Z^{\lambda}}$ the corresponding perverse sheaves on $Z^{\lambda}/Z_G$ and $\mathring{\mathfrak{ch}} \times Z^{\lambda}$. In particular we have the perverse sheaf $\Psi(\mathscr{W}_{Z^{\lambda}})$ on $Z^{\lambda}$.

Recall that in the introduction we chose a dominant regular cocharacter $\gamma : \mathbb{G}_m \to T$. Denote by $^{\prime}a_{\gamma} : \mathbb{G}_m \times Z \to Z$ the resulting action morphism. Then $\Psi( \, \! ^{\prime}a_{\gamma}^{\Delta}\mathscr{W}_{Z^{\lambda}})$ is canonically isomorphic to the previously constructed perverse sheaf $\Psi(\mathscr{W}_{Z^{\lambda}})$, where in the former expression nearby cycles is taken with respect to the projection $\A^1 \times Z^{\lambda} \to \A^1$. In particular, the choice of $\gamma$ gives rise to a distinguished monodromy endomorphism for $\Psi(\mathscr{W}_{Z^{\lambda}})$.

\begin{theorem}

For any $\mu \in \Lambda^{\pos}$ there is an isomorphism \[ ^{\prime}\mathfrak{j}^!_{=\mu}\Psi( \mathscr{W}_Z) \tilde{\longrightarrow} ^{\prime}\mathfrak{m}^{\Delta}_{\mu}\Omega^{\mu}. \]

\label{zstrata}

\end{theorem}

We will need to use the factorization structure on $Z$ in the following way. First, observe that $\mathscr{W}_Z$ admits a natural factorization structure. Thus $\Psi( \mathscr{W}_Z)$ admits a factorization structure by the K\"{u}nneth formula for nearby cycles. Although the K\"{u}nneth formula holds for the total nearby cycles functor, in this case the total nearby cycles equals the unipotent nearby cycles because $\mathscr{W}_Z$ is $\G_m$-equivariant.

\subsection{} In the second proof, Theorems \ref{strata} and \ref{zstrata} will be proved simultaneously by an inductive argument. The argument uses the following key lemma.

Let $f : \mathscr{X} \to \mathscr{Y}$ be a morphism of Artin stacks with $\mathscr{Y}$ smooth, and suppose we are given a function $\mathscr{Y} \to \A^1$. Let $g : S \to \mathscr{Y}$ be a morphism where $S$ is an affine scheme and consider the cartesian square \[
\begin{tikzcd}[cramped]
\mathscr{X} \times_{\mathscr{Y}} S \ar[d, "^{\prime \!}f"] \ar[r, "^{\prime \!}g"] & \mathscr{X}\ar[d, "f"] \\
S \ar[r, "g"] & \mathscr{Y}.
\end{tikzcd} \]
Write $i : S_0 \to S$ for the inclusion of the vanishing locus of the function $S \to \mathscr{Y} \to \A^1$.

\begin{lemma}

For any $\mathscr{F} \in D(\mathscr{X})$ which is ULA over $\mathscr{Y}$ and any $\mathscr{G} \in D(S)$, there is a canonical isomorphism \[ \Psi(^{\prime \!}g^!(\mathscr{F}) \overset{!}{\otimes} \, \! ^{\prime \!}f^!(\mathscr{G})) \, \tilde{\longrightarrow} \, ^{\prime \!}g^!(\mathscr{F}) \overset{!}{\otimes} \, \! ^{\prime \!}f^!(i^!\Psi(\mathscr{G})). \]

\label{nclem}

\end{lemma}

Fix $\nu \in \Lambda^{\pos}$. We will apply the lemma in the case $\mathscr{X} = \A^1 \times \overline{\Bun}_{B^-,\leq \nu}^{\lambda}$, $\mathscr{F} = \IC_{\A^1} \boxtimes \mathfrak{j}^-_!(\IC_{\Bun_{B^-}^{\lambda}})|_{\overline{\Bun}_{B^-,\leq \nu}^{\lambda}}$, and $\mathscr{Y} = \A^1 \times \Bun_G$. Let us check that the ULA property holds when $\lambda$ is sufficiently dominant relative to $\nu$. According to Corollary 4.5 in \cite{BG2}, we have the following decomposition in the Grothendieck group: \[ [\mathfrak{j}^-_!(\IC_{\Bun_{B^-}})|_{\overline{\Bun}_{B^-,\leq \nu}^{\lambda}}] = \sum_{\eta \leq \nu} [\mathfrak{j}^-_{\geq \eta,!}(\Omega^{\eta} \boxtimes \IC_{\overline{\Bun}_{B^-,\leq \nu - \eta}^{\lambda - \eta}})]. \] Observe that the diagram \[
\begin{tikzcd}[cramped]
X^{\eta} \times \overline{\Bun}_{B^-}^{\lambda-\eta} \ar[d] \ar[r, "\mathfrak{j}^-_{\geq \eta}"] & \overline{\Bun}_{B^-}^{\lambda} \ar[d] \\
\overline{\Bun}_B^{\lambda-\eta} \ar[r] & \Bun_G,
\end{tikzcd} \]
commutes, where the left vertical arrow is projection onto the second factor. Since $\mathfrak{j}^-_{\geq \eta}$ is proper it suffices to prove that, for $\lambda$ sufficiently dominant, $\IC_{\overline{\Bun}_{B^-,\leq \nu - \eta}^{\lambda-\eta}}$ is ULA over $\Bun_G$ for all $\eta \leq \nu$. This follows from Corollary 4.1.1.1 in \cite{C2}.

\begin{proof}[Second proof of Theorems \ref{strata} and \ref{zstrata}]

Observe that Theorem \ref{strata} is trivial on the open stratum, since \[ \Psi(\mathscr{W})|_{\Bun_{N^{\omega}}} = \Psi(\mathscr{W}|_{\mathbb{G}_m \times \Bun_{N^{\omega}}}) = \Psi(\IC_{\mathbb{G}_m} \boxtimes \IC_{\Bun_{N^{\omega}}}) = \IC_{\Bun_{N^{\omega}}}, \] and similarly for Theorem \ref{zstrata} on $\mathring{Z}$.

We begin by proving Theorem \ref{zstrata} for the deepest strata, i.e. the closed embeddings \[ ^{\prime}\mathfrak{j}_{=\mu}^{\mu} : X^{\mu} \longrightarrow Z^{\mu}. \] Recall that $\Psi( \mathscr{W}_Z)$ is $\mathbb{G}_m$-monodromic by construction, so the contraction principle says that \[ ^{\prime}\mathfrak{j}_{=\mu}^{\mu,!}\Psi( \mathscr{W}_Z) = \pi^{\mu}_!\Psi( \mathscr{W}_Z). \]

Write $\mathring{\pi}^{\mu} := \pi^{\mu} \circ \, \! ^{\prime}\mathfrak{j}^{\mu}_{=0}$. Theorem 4.6.1 in \cite{R} implies that there is an isomorphism \[ (\id_{\mathbb{G}_m} \times \mathring{\pi})_!( \mathscr{W}_Z|_{\mathbb{G}_m \times \mathring{Z}}) \tilde{\longrightarrow} \IC_{\mathbb{G}_m} \boxtimes \Omega, \] compatible with the factorization structures. Since $\mathscr{W}_Z$ is $!$-extended from $\mathbb{G}_m \times \mathring{Z}$, we obtain  \[ (\id_{\mathbb{G}_m} \times \pi)_! \! \, \mathscr{W}_Z \tilde{\longrightarrow} \IC_{\mathbb{G}_m} \boxtimes \Omega. \]

Since $\pi = \overline{\pi} \circ \! \, ^{\prime}\mathfrak{j}^-$ and $\Psi$ commutes with proper pushforwards, we have \[ \overline{\pi}_!\Psi((\id_{\mathbb{G}_m} \times \! \,^{\prime}\mathfrak{j}^-)_!\! \, \mathscr{W}_{Z}) \tilde{\longrightarrow} \Omega. \] Therefore it suffices to prove that the canonical morphism \[ ^{\prime}\mathfrak{j}^-_!\Psi( \! \, \mathscr{W}_Z) \longrightarrow \Psi((\id_{\mathbb{G}_m} \times \! \, ^{\prime}\mathfrak{j}^-)_! \! \, \mathscr{W}_Z) \] is an isomorphism. Since $\Psi$ commutes with Verdier duality we can replace the $!$-pushforwards with $*$-pushforwards.

Fix $S \to \overline{\Bun}_{N^{\omega}}$ with $S$ an affine scheme and apply Lemma \ref{nclem} with $f = \id_{\A^1} \times \overline{\mathfrak{p}}^-$, \[ g : \A^1 \times S \longrightarrow \A^1 \times \overline{\Bun}_{N^{\omega}} \longrightarrow \A^1 \times \Bun_G, \] $\mathscr{F} = \IC_{\A^1} \boxtimes \mathfrak{j}^-_*(\omega_{\Bun_{B^-}})|_{\overline{\Bun}_{B^-,\leq \nu}^{\lambda}}$, and $\mathscr{G} = \mathscr{W}|^!_{\A^1 \times S}$. Then the lemma yields an isomorphism \[ \Psi((\id_{\mathbb{G}_m} \times \! \, ^{\prime}\mathfrak{j}^-)_* \! \, \mathscr{W}_Z)|_{_{\leq \nu}\overline{Z}^{\lambda}} \tilde{\longrightarrow} \ ^{\prime}\mathfrak{j}^-_*\Psi( \! \, \mathscr{W}_Z)|_{_{\leq \nu}\overline{Z}^{\lambda}} \] for $\lambda$ sufficiently dominant. Changing $\lambda$ if necessary so that $\lambda \geq \mu$, we can restrict this isomorphism along the map \[ (_{\leq \nu}\overline{Z}^{\mu} \times \mathring{Z}^{\lambda - \mu})_{\disj} \longrightarrow \ \! _{\leq \nu}\overline{Z}^{\lambda}. \] By factorizability we obtain the desired isomorphism on $_{\leq \nu}\overline{Z}^{\mu}$. Since $\nu$ was arbitrary, Theorem \ref{zstrata} is proved for the deepest strata.

Now we prove Theorem \ref{strata}. Fix $\mu \in \Lambda^{\pos}$ and choose $\lambda \geq \mu$ dominant enough that $Z^{\lambda}_{=\mu}$ surjects smoothly onto $X^{\mu} \times_{\Bun_T} \Bun_B$. Note that $(X^{\mu} \times \mathring{Z}^{\lambda - \mu})_{\disj}$ is one of the connected components of the fiber product \[ (X^{\mu} \times X^{\lambda - \mu})_{\disj} \times_{X^{\lambda}} Z^{\lambda}_{=\mu}, \] and that the former surjects onto $\overline{\Bun}_{N^{\omega},=\mu}$. Theorem \ref{zstrata} for the deepest and open strata implies that the cohomologically normalized pullback of $\mathfrak{j}_{=\mu}^!\Psi(\mathscr{W})$ to $(X^{\mu} \times \mathring{Z}^{\lambda - \mu})_{\disj}$ is $\Omega^{\mu} \boxtimes \IC_{\mathring{Z}^{\lambda - \mu}}$. Theorem \ref{strata} follows once we observe that the composition \[ (X^{\mu} \times \mathring{Z}^{\lambda - \mu})_{\disj} \longrightarrow \overline{\Bun}_{N^{\omega},=\mu} \stackrel{\mathfrak{m}_{\mu}}{\longrightarrow} X^{\mu} \] is the projection onto the first factor and apply Lemma \ref{adelicinv}.

The previous paragraph implies Theorem \ref{zstrata} holds on the stratum $Z_{=\mu}^{\lambda}$. Let $\lambda' \geq \mu$ and change $\lambda$ if necessary so that $\lambda \geq \lambda'$. By restricting along the morphism \[ (Z^{\lambda'}_{=\mu} \times \mathring{Z}^{\lambda - \lambda'})_{\disj} \longrightarrow Z^{\lambda}_{=\mu} \] and invoking factorization, we obtain Theorem \ref{zstrata}.

The remainder of Theorems \ref{strata} and \ref{zstrata} follows as in the first proof.

\end{proof}

\section{First proof of Theorem \ref{main}}

\subsection{} Like Theorem \ref{strata}, we formulate the analogue of Theorem \ref{main} on Zastava space. 

\begin{theorem}

For any $\lambda \in \Lambda^{\pos}$, there is an $\mathfrak{sl}_2$-equivariant isomorphism of factorizable sheaves
\begin{equation}
\gr \Psi( \, \! \mathscr{W}_{Z^{\lambda}}) \ \tilde{\longrightarrow} \bigoplus_{0 \leq \mu \leq \lambda} \ \! ^{\prime}\mathfrak{j}_{=\mu,!*}^{\lambda} \ \! \! ^{\prime}\mathfrak{m}^{\lambda,\Delta}_{\mu}\mathscr{P}^{\mu}.
\label{zmaineq}
\end{equation}

\label{zmain}

\end{theorem}

Now we work out three of the simplest cases of Theorem \ref{zmain}. For brevity, we will write $\Psi := \Psi( \, \! \mathscr{W}_{Z^{\lambda}})$.

\begin{example}

\label{simpex}

Let $\alpha$ be a simple coroot. There is an isomorphism $Z^{\alpha} \cong X \times \A^1$ under which $\mathring{Z}^{\alpha} \cong X \times \mathring{\A}^1$, where $\mathring{\A}^1 := \A^1 \setminus \{ 0 \}$. The canonical map $\mathring{Z}^{\alpha} \to \A^1$ is given in these terms by $(x,t) \mapsto \frac{1}{t}$. It follows from Example 4.3 of \cite{C1} that $\Psi$ is the cohomologically normalized pullback of the unique indecomposable tilting sheaf on $\A^1$ which extends $\IC_{\mathring{\A}^1}$. Moreover, the monodromy filtration \[ F_{-1} \subset F_0 \subset F_1 = \Psi \] satisfies $F_{-1} \cong \IC_{Z^{\alpha}_{=\alpha}}$, $F_0/F_{-1} \cong \IC_{Z^{\alpha}}$, and $F_1/F_0 \cong \IC_{Z^{\alpha}_{=\alpha}}$, where $Z^{\alpha}_{=\alpha} = X \times \{ 0 \} \subset X \times \A^1 = Z^{\alpha}$. This confirms Theorem \ref{zmain} in the case $\lambda = \alpha$.

\end{example}

\begin{example}

\label{2simpex}

Now consider the case $\lambda = 2\alpha$. By Theorem \ref{zstrata} for $\mu = 2\alpha$, we have a short exact sequence \[ \Omega^{2\alpha} \longrightarrow \Psi \longrightarrow \! \, ^{\prime}\mathfrak{j}^{2\alpha}_{\leq \alpha,*} \! \, ^{\prime}\mathfrak{j}^{2\alpha,*}_{\leq \alpha}\Psi \] (recall that $\Omega^{2\alpha}$ is the clean extension of the sign local system on $X^{(2)}_{\disj}$). Similarly, applying Theorem \ref{zstrata} for $\mu = \alpha$ and $0$ we obtain an exact triangle \[ ^{\prime}\mathfrak{j}^{2\alpha}_{=\alpha,*}\IC_{Z^{2\alpha}_{=\alpha}} \longrightarrow \! \, ^{\prime}\mathfrak{j}^{2\alpha}_{\leq \alpha,*} \! \, ^{\prime}\mathfrak{j}^{2\alpha,*}_{\leq \alpha}\Psi \longrightarrow \! \, ^{\prime}\mathfrak{j}^{2\alpha}_{=0,*}\IC_{\mathring{Z}^{2\alpha}}, \] where we used the fact that $\Omega^{\alpha} \cong \IC_X$. Applying Verdier duality to the equation in Corollary 4.5 of \cite{BG2} (or rather the analoguous equation on Zastava space), we have \[ [ \! \, ^{\prime}\mathfrak{j}^{2\alpha}_{=0,*}\IC_{\mathring{Z}^{2\alpha}}] = [\IC_{Z^{2\alpha}}] + [\IC_{\overline{Z^{2\alpha}_{=\alpha}}}] + [\Omega^{2\alpha}], \] where we identified $\Omega^{2\alpha}$ with its Verdier dual. Finally, one computes the simple constituents of $^{\prime}\mathfrak{j}^{2\alpha}_{=\alpha,*}\IC_{Z^{2\alpha}_{=\alpha}}$ as follows: first, consider the short exact sequence \[ \IC_{\overline{Z^{2\alpha}_{=\alpha}}} \longrightarrow \! \, ^{\prime}\mathfrak{j}^{2\alpha}_{=\alpha,*}\IC_{Z^{2\alpha}_{=\alpha}} \longrightarrow \! \, ^{\prime}\mathfrak{j}^{2\alpha,!}_{=2\alpha}\IC_{\overline{Z^{2\alpha}_{=\alpha}}}[1]. \] To compute the third term, observe that there is a Cartesian square \[
\begin{tikzcd}[cramped]
X^2 \ar[d] \ar[r] & Z^{2\alpha}_{\geq \alpha} \ar[d, "^{\prime}\mathfrak{j}^{2\alpha}_{\geq \alpha}"] \\
X^{(2)} \ar[r, "^{\prime}\mathfrak{j}^{2\alpha}_{=2\alpha}"] & Z^{2\alpha},
\end{tikzcd} \]
and that the $!$-restriction of $\IC_{Z^{2\alpha}_{\geq \alpha}}$ along the top horizontal morphism is $\IC_{X^2}[-1]$. Since $^{\prime}\mathfrak{j}^{2\alpha}_{\geq \alpha}$ is finite and birational onto its image, we can use base change to compute \[ ^{\prime}\mathfrak{j}^{2\alpha,!}_{=2\alpha}\IC_{\overline{Z^{2\alpha}_{=\alpha}}}[1] = \! \, ^{\prime}\mathfrak{j}^{2\alpha,!}_{=2\alpha} \! \, ^{\prime}\mathfrak{j}^{2\alpha}_{\geq \alpha,*}\IC_{Z^{2\alpha}_{\geq \alpha}}[1] = \IC_{Z^{2\alpha}_{=2\alpha}} \! \oplus \Omega^{2\alpha}. \] Summarizing, we have \[ [\Psi] = [\IC_{Z^{2\alpha}}] + 2[\IC_{\overline{Z^{2\alpha}_{=\alpha}}}] + [\IC_{Z^{2\alpha}_{=2\alpha}}] + 3[\Omega^{2\alpha}]. \] 

Now we will determine which graded component of $\gr \Psi$ each simple subquotient lies in. In what follows, ``weight" refers to an eigenvalue of the Lefschetz Cartan operator. Since the monodromy filtration is compatible with the factorization structure, when we pull back $\gr \Psi$ along the factorization map \[ (Z^{\alpha} \times Z^{\alpha})_{\disj} \longrightarrow Z^{2\alpha} \] we get $(\gr \Psi( \, \! \mathscr{W}_{Z^{\alpha}}))^{\boxtimes 2}$ restricted to $(Z^{\alpha} \times Z^{\alpha})_{\disj}$. By the previous example, the latter sheaf with $\mathfrak{sl}_2$-action is isomorphic to
\begin{equation}
(\std^{\otimes 2} \otimes \IC_{X^2}) \oplus (\std \otimes \IC_{X \times Z^{\alpha}}) \oplus (\std \otimes \IC_{Z^{\alpha} \times X}) \oplus (\triv \otimes \IC_{Z^{\alpha} \times Z^{\alpha}}).
\label{rk2fact}
\end{equation}
It follows immediately that $\IC_{Z^{2\alpha}}$ has weight $0$ and that the two copies of $\IC_{\overline{Z^{2\alpha}_{=\alpha}}}$ have weights $\pm 1$. Since $\std^{\otimes 2} \cong V_2 \oplus \triv$, the three copies of $\Omega^{2\alpha}$ have weights $-2$, $0$, and $2$, and $\IC_{Z^{2\alpha}_{=2\alpha}}$ has weight $0$.

In terms of the monodromy filtration \[ F_{-2} \subset F_{-1} \subset F_0 \subset F_1 \subset F_2 = \Psi, \] we have $F_{-2} \cong \Omega^{2\alpha} \cong F_2/F_0$, $F_{-1}/F_{-2} \cong \IC_{\overline{Z^{2\alpha}_{=\alpha}}} \cong F_1/F_0$, and $F_0/F_{-1}$ has simple constituents $\IC_{Z^{2\alpha}}$, $\IC_{Z^{2\alpha}_{=2\alpha}}$, and $\Omega^{2\alpha}$. So in order to prove Theorem \ref{zmain} in the case $\lambda = 2\alpha$, it remains to show that $F_0/F_{-1}$ is semisimple. Its pullback along the factorization map is semisimple, and since semisimplicity is \'{e}tale local, the restriction of $F_0/F_{-1}$ to $Z^{2\alpha} \setminus \pi^{-1}(X)$ is semisimple. But $\Psi$ has no simple subquotients supported on $\pi^{-1}(X)$, so $F_0/F_{-1}$ is the intermediate extension of its restriction to $Z^{2\alpha} \setminus \pi^{-1}(X)$ and therefore semisimple.

\end{example}

\begin{example}

\label{sumsimpex}

Suppose $\lambda = \alpha + \beta$ is a coroot, where $\alpha$ and $\beta$ are distinct simple coroots. Applying Theorem \ref{zstrata} for $\mu = \alpha + \beta$, we obtain the short exact sequence \[ \Omega^{\lambda} \longrightarrow \Psi \longrightarrow \! \, ^{\prime}\mathfrak{j}^{\lambda}_{< \lambda,*} \! \, ^{\prime}\mathfrak{j}^{\lambda,*}_{< \lambda}\Psi. \] Similar considerations yield the short exact sequence \[ ^{\prime}\mathfrak{j}^{\lambda}_{=\alpha,*}\IC_{Z^{\lambda}_{=\alpha}} \oplus \! \, ^{\prime}\mathfrak{j}^{\lambda}_{=\beta,*}\IC_{Z^{\lambda}_{=\beta}} \longrightarrow \! \, ^{\prime}\mathfrak{j}^{\lambda}_{< \lambda,*} \! \, ^{\prime}\mathfrak{j}^{\lambda,*}_{< \lambda}\Psi \longrightarrow \! \, ^{\prime}\mathfrak{j}^{\lambda}_{=0,*}\IC_{\mathring{Z}^{\lambda}}. \] Applying Corollary 4.5 of \cite{BG2}, we have \[ [ \! \, ^{\prime}\mathfrak{j}^{\lambda}_{=0,*}\IC_{\mathring{Z}^{\lambda}}] = [\IC_{Z^{\lambda}}] + [\IC_{\overline{Z^{\lambda}_{=\alpha}}}] + [\IC_{\overline{Z^{\lambda}_{=\beta}}}] + [\Upsilon^{\lambda}], \] where $\Upsilon^{\lambda}$ is the Verdier dual of $\Omega^{\lambda}$. According to Section 1.3.2 in \cite{AG}, in this case $\Omega^{\lambda}$ is the $*$-extension of $\IC_{X^2_{\disj}}$ to $X^{\lambda} = X^2$, whence $\Upsilon^{\lambda}$ is the $!$-extension. In particular $[\Omega^{\lambda}] = [\Upsilon^{\lambda}] = [\IC_{X^2}] + [\IC_X]$. As for the remaining simple constituents, consider the short exact sequence \[ \IC_{\overline{Z^{\lambda}_{=\alpha}}} \longrightarrow \! \, ^{\prime}\mathfrak{j}^{\lambda}_{=\alpha,*}\IC_{Z^{\lambda}_{=\alpha}} \longrightarrow \! \, ^{\prime}\mathfrak{j}^{\lambda,!}_{=\alpha}\IC_{\overline{Z^{\lambda}_{=\alpha}}}[1], \] and similarly for $\beta$. The third term is $\IC_{X^2}$, so finally we see that \[ [\Psi] = [\IC_{Z^{\lambda}}] + 2[\IC_{\overline{Z^{\lambda}_{=\alpha}}}] + 2[\IC_{\overline{Z^{\lambda}_{=\beta}}}] + 4[\IC_{X^2}] + 2[\IC_X]. \]

Now we compute the weights of the simple subquotients of $\Psi$. We have the factorization morphism \[ (Z^{\alpha} \times Z^{\beta})_{\disj} \longrightarrow Z^{\alpha+\beta}, \] and after pulling back $\gr \Psi$ the result is (\ref{rk2fact}), up to relabeling $\beta$ as $\alpha$. As in Example \ref{2simpex}, it follows that $\IC_{Z^{\alpha+\beta}}$ has weight $0$, the two copies of $\IC_{\overline{Z^{\lambda}_{=\alpha}}}$ have weights $\pm 1$ and likewise for $\IC_{\overline{Z^{\lambda}_{=\beta}}}$, and the four copies of $\IC_{X^2}$ have weights $-2$, $0$, $0$, and $2$. We will see below that in any case where $\lambda$ is a coroot, there are two simple subquotients of $\Psi$ isomorphic to $\IC_X$, with weights $\pm 1$. In terms of the monodromy filtration, we have $F_{-2} \cong \IC_{X^2} \cong F_2/F_1$, $F_{-1}/F_{-2}$ and $F_1/F_0$ each have simple subquotients $\IC_{\overline{Z^{\lambda}_{=\alpha}}}$, $\IC_{\overline{Z^{\lambda}_{=\beta}}}$, and $\IC_X$, and $F_0/F_{-1}$ has simple subquotients $\IC_{Z^{\lambda}}$ and $\IC_{X^2}$, the latter with multiplicity two. As in Example \ref{2simpex}, one uses factorization to show that $F_0/F_{-1}$ is semisimple. To prove the semisimplicity of $F_{-1}/F_{-2}$ and $F_1/F_0$, it is enough to show that there are no extensions between $\IC_X$ and $\IC_{\overline{Z^{\lambda}_{=\alpha}}}$. By Verdier duality, it suffices to prove that \[ \Ext^1_{D(Z^{\lambda})}(\IC_X,\IC_{\overline{Z^{\lambda}_{=\alpha}}}) = 0. \] We have $^{\prime}\mathfrak{j}^{\lambda,!}_{=\lambda}\IC_{\overline{Z^{\lambda}_{=\alpha}}} = \IC_{X^2}[-1]$, whence $\Delta^!\IC_{\overline{Z^{\lambda}_{=\alpha}}} = \IC_X[-2]$. Thus \[ \RHom_{D(Z^{\lambda})}(\IC_X,\IC_{\overline{Z^{\lambda}_{=\alpha}}}) = H^{\bullet}(X)[-2], \] and in particular $\Ext^1$ vanishes.

\end{example}

\subsection{} The following lemma will be used in both proofs of Theorems \ref{main} and \ref{zmain}. For any $\mu \in \Lambda^{\pos}$ write \[ \Delta^{\mu} : X \times_{\Bun_T} \overline{\Bun}_B \longrightarrow \overline{\Bun}_{N^{\omega}} \] for the finite birational map defined as the composition of $\mathfrak{j}_{\geq \mu}$ and the embedding \[ X \times_{\Bun_T} \overline{\Bun}_B = X \times_{X^{\mu}} \overline{\Bun}_{N^{\omega},\geq \mu} \longrightarrow \overline{\Bun}_{N^{\omega},\geq \mu} \] induced by the diagonal map $X \to X^{\mu}$.

\begin{lemma}

If $\mu$ is a coroot, then there is an indecomposable subquotient $\mathscr{M}$ of $\Psi(\mathscr{W})$ with a filtration \[ \mathscr{M}_{-1} \subset \mathscr{M}_0 \subset \mathscr{M}_1 = \mathscr{M} \] such that $\mathscr{M}_{-1} \cong \Delta^{\mu}_*\IC_{X \times_{\Bun_T} \overline{\Bun}_B}$, $\mathscr{M}_0/\mathscr{M}_{-1} \cong \IC_{\overline{\Bun}_{N^{\omega}}}$, and $\mathscr{M}/\mathscr{M}_0 \cong \Delta^{\mu}_*\IC_{X \times_{\Bun_T} \overline{\Bun}_B}$.

\label{threestep}

\end{lemma}

\begin{proof}

First, we claim the subsheaf $\mathfrak{j}_{=0,!}\IC_{\Bun_{N^{\omega}}}$ of $\Psi(\mathscr{W})$ has a quotient $\mathscr{M}_0$ of the form described above. Recall that $\mathfrak{j}_{=0,!}\IC_{\Bun_{N^{\omega}}}$ has a descending filtration with subquotients $\mathfrak{j}_{=\nu,!*}\mathfrak{r}^{\Delta}_{\nu}\Omega^{\nu}$, and in particular has $\mathfrak{j}_{=\mu,!*}\Omega^{\mu}$ as a subquotient and $\IC_{\Bun_{N^{\omega}}}$ as a quotient. The former sheaf has $\Delta^{\mu}_*\IC_{X \times_{\Bun_T} \overline{\Bun}_B}$ as a quotient because $\mu$ is a coroot, so it suffices to show that that for any $0 < \nu < \mu$ we have \[ \Ext^1(\Delta^{\mu}_*\IC_{X \times_{\Bun_T} \overline{\Bun}_B},\mathfrak{j}_{=\nu,!*}\mathfrak{r}^{\Delta}_{\nu}\Omega^{\nu}) = 0. \] One computes using base change that \[ \Delta^{\mu,!}\mathfrak{j}_{=\nu,!*}\mathfrak{r}^{\Delta}_{\nu}\Omega^{\nu} = \Delta^!\Omega^{\nu} \stackrel{!}{\otimes} \Delta^{\mu,!}\IC_{\overline{\Bun}_{N^{\omega}}}. \] Since $\Delta^{\mu,!}\IC_{\overline{\Bun}_{N^{\omega}}}$ is concentrated in cohomological degrees $\geq 1$, and both $\Delta^!\Omega^{\nu}$ and $\Delta^{\mu,!}\IC_{\overline{\Bun}_{N^{\omega}}}$ have lisse (actually constant) cohomology sheaves, their $!$-tensor product is concentrated in cohomological degrees $\geq 2$. It follows that the $\Ext^1$ above vanishes.

We have shown that $\mathfrak{j}_{=0,!}\IC_{\Bun_{N^{\omega}}}$ has a quotient $\mathscr{M}_0$ which fits into a short exact sequence \[ \Delta^{\mu}_*\IC_{X \times_{\Bun_T} \overline{\Bun}_B} \longrightarrow \mathscr{M}_0 \longrightarrow \IC_{\overline{\Bun}_{N^{\omega}}}. \] This sequence does not split because $\mathscr{M}_0$ is the quotient of the indecomposable sheaf $\mathfrak{j}_{=0,!}\IC_{\Bun_{N^{\omega}}}$ with simple cosocle $\IC_{\overline{\Bun}_{N^{\omega}}}$. Dually, we obtain a subsheaf $\mathscr{M}/\mathscr{M}_{-1}$ of $\mathfrak{j}_{=0,*}\IC_{\Bun_{N^{\omega}}}$ of the desired form, from which follows the existence of $\mathscr{M}$.

\end{proof}

We will first give a proof of Theorem \ref{main} under the assumption that $\gr \Psi(\mathscr{W})$ is semisimple. The semisimplicity can be proved via Mochizuki's theory of weights for holonomic $D$-modules, since $\mathscr{W}$ is pure and (up to shift) the monodromy filtration on nearby cycles of a pure sheaf coincides with the weight filtration (see Corollary 9.1.10 in \cite{M}).

\begin{proof}[First proof of Theorems \ref{main} and \ref{zmain}]

Both sides of the isomorphism (\ref{maineq}) restrict to $\IC_{\Bun_{N^{\omega}}}$. Suppose that we have constructed the isomorphism over $\overline{\Bun}_{N^{\omega},< \mu}$. Then for $\lambda \geq \mu$ sufficiently dominant, pulling back along $^{\prime}\mathfrak{p}^-$ yields the isomorphism (\ref{zmaineq}) over $Z^{\lambda}_{<\mu}$. One obtains (\ref{zmaineq}) on $Z^{\mu}_{<\mu}$ by pullback along the factorization map \[ (Z^{\mu}_{<\mu} \times \mathring{Z}^{\lambda-\mu})_{\disj} \longrightarrow Z^{\lambda}_{<\mu}, \] since the inverse images of both sides of the isomorphism factorize and are constant along the second component. The same argument yields (\ref{zmaineq}) on $Z^{\mu'}_{<\mu} = Z^{\mu'}$ for $\mu' < \mu$.

On the other hand, one can use factorization to obtain (\ref{zmaineq}) on $Z^{\mu} \setminus \pi^{-1}(X)$. Namely, for $\mu_1 + \mu_2 = \mu$, $\mu_1,\mu_2 < \mu$, the pullback of both sides of the isomorphism along \[ (Z^{\mu_1} \times Z^{\mu_2})_{\disj} \longrightarrow Z^{\mu} \] are identified. Since this factorization map is \'{e}tale but not necessarily an embedding, we must argue that the isomorphism descends to its image. This immediately reduces to the case that $\mu = n \cdot \alpha$ for some $\alpha \in \Delta$. In this case, both sides of (\ref{zmaineq}) are the intermediate extension of their restriction to $\pi^{-1}(X^{\mu}_{\disj})$, so it suffices to show that the isomorphism over $(Z^{\alpha})^n_{\disj}$ descends to $\pi^{-1}(X^{\mu}_{\disj}) \subset Z^{\mu}$. Since $(Z^{\alpha})^n_{\disj}$ is a $\Sigma_n$-torsor over $\pi^{-1}(X^{\mu}_{\disj})$ and $\Sigma_n$ is generated by transpositions, the claim reduces to the case $n = 2$. But this was already done in Example \ref{2simpex}.

Note that $Z^{\mu}_{<\mu} \cup (Z^{\mu} \setminus \pi^{-1}(X)) = Z^{\mu} \setminus \Delta(X)$. The isomorphisms of the previous two paragraphs clearly agree on $Z^{\mu}_{<\mu} \cap (Z^{\mu} \setminus \pi^{-1}(X))$, hence glue to an isomorphism away from the main diagonal.

If $\mu$ is not a coroot, then we claim that $\Psi( \, \! \mathscr{W}_{Z^{\mu}})$ has no simple subquotients supported on the main diagonal, whence $\gr \Psi(\mathscr{W}_{Z^{\mu}})$ is the intermediate extension of its restriction to $Z^{\mu} \setminus \Delta(X)$. This is true for the right hand side of (\ref{zmaineq}) by construction, so the claim implies that the isomorphism extends to $Z^{\mu}$ in this case. By Theorem \ref{zstrata} there is a filtration of $\Psi( \, \! \mathscr{W}_{Z^{\mu}})$ by the sheaves $^{\prime}\mathfrak{j}^{\mu}_{=\nu,*} \ \! ^{\prime}\mathfrak{m}^{\mu,\Delta}_{\nu}\Omega^{\nu}$ for $0 \leq \nu \leq \mu$. Using Corollary 4.5 of \cite{BG2}, one can show that $^{\prime}\mathfrak{j}^{\mu}_{=\nu,*} \ \! ^{\prime}\mathfrak{m}^{\mu,\Delta}_{\nu}\Omega^{\nu}$ surjects onto $^{\prime}\mathfrak{j}^{\mu}_{=\mu,*}\add_*(\Omega^{\nu} \boxtimes \Upsilon^{\mu - \nu})$, and that no subquotient of the kernel is supported on $X^{\mu}$. Now the claim follows, because out of the latter sheaves only $\Omega^{\mu}$ and $\Upsilon^{\mu}$ could have subquotients supported on the diagonal, and by Section 3.3 of \emph{loc. cit.} this occurs if and only if $\mu$ is a coroot.

Suppose that $\mu$ is a coroot. Then $\Delta_*\IC_X$ appears as a subquotient of $\Omega^{\mu}$ and of $\Upsilon^{\mu}$ with multiplicity one. By the analysis in the previous paragraph $\Delta_*\IC_X$ appears as a summand of $\gr \Psi( \, \! \mathscr{W}_{Z^{\mu}})$ with multiplicity two, and there are no other subquotients supported on the main diagonal. Thus the isomorphism (\ref{zmaineq}) extends to $Z^{\mu}$, and it remains to show that $\mathfrak{sl}_2$ acts on the summand $\IC_X^{\oplus 2}$ of $\gr \Psi( \, \! \mathscr{W}_{Z^{\mu}})$ as the standard representation. 

The only other possibility is that $\mathfrak{sl}_2$ acts on $\IC_X^{\oplus 2}$ trivially, which would imply that the subquotient $\mathscr{M}$ from Lemma \ref{threestep} is a subquotient of $F_0/F_{-1}$. But $\mathscr{M}$ is indecomposable and $F_0/F_{-1}$ is semisimple, so this is impossible.

Having constructed the isomorphism of Theorem \ref{zmain} over $Z^{\mu}$, we can complete the inductive step of Theorem \ref{main} by extending the isomorphism from $\overline{\Bun}_{N^{\omega},< \mu}$ to $\overline{\Bun}_{N^{\omega},\leq \mu}$. Choose $\lambda \geq \mu$ dominant enough that $Z^{\lambda}_{\leq \mu}$ surjects smoothly onto $\overline{\Bun}_{N^{\omega},\leq \mu}$. As in the proof of Theorem \ref{strata}, note that $(Z^{\mu} \times \mathring{Z}^{\lambda - \mu})_{\disj}$ is one of the connected components of the fiber product \[ (X^{\mu} \times X^{\lambda - \mu})_{\disj} \times_{X^{\lambda}} Z^{\lambda}_{\leq \mu}, \] and that the former surjects onto $\overline{\Bun}_{N^{\omega},\leq \mu}$. By factorization, the cohomologically normalized pullback of $\gr \Psi(\mathscr{W})$ to $(Z^{\mu} \times \mathring{Z}^{\lambda - \mu})_{\disj}$ is $(\gr \Psi(\mathscr{W}_{Z^{\mu}})) \boxtimes \IC_{\mathring{Z}^{\lambda - \mu}}$.

A factorization argument as in the proof of Theorem \ref{zstrata} allows us to construct the isomorphism (\ref{zmaineq}) over $Z^{\lambda}_{\leq \mu}$ for arbitrary $\lambda \geq \mu$, which completes the proof of Theorem \ref{zmain}.

\end{proof}

\begin{proof}[Proof of Corollary \ref{kermon}]

It suffices to prove the corresponding claim on $Z^{\lambda}$ for any $\lambda \in \Lambda^{\pos}$. The morphism \[ ^{\prime}\mathfrak{j}^{\lambda}_{=0,!}\IC_{\Bun_{N^{\omega}}} \longrightarrow \Psi( \! \, \mathscr{W}) \] is injective because $\Psi( \! \, \mathscr{W})$ is tilting, so it suffices to show that $^{\prime}\mathfrak{j}^{\lambda}_{=0,!}\IC_{\Bun_{N^{\omega}}}$ and the kernel of monodromy have the same class in the Grothendieck group. By factorization and induction this holds away from the main diagonal, and Corollary 4.5 of \cite{BG2} implies that the only subquotient of $^{\prime}\mathfrak{j}^{\lambda}_{=0,!}\IC_{\Bun_{N^{\omega}}}$ supported on the main diagonal is $\Delta_*\IC_X$ with multiplicity one. Theorem \ref{zmain} implies that the same is true for the kernel of the monodromy operator on $\Psi( \! \, \mathscr{W})$.

\end{proof}

\section{Second proof of Theorem \ref{main}}

\subsection{} In this section we will give a proof of Theorems \ref{main} and \ref{zmain} which does not use weights for irregular holonomic $D$-modules to prove the semisimplicity of $\gr \Psi(\mathscr{W})$, but instead depends on Conjecture \ref{dualact} (but only the part proved in \cite{FFKM}). First we make the statement of the conjecture more precise in the case of a trivial $\check{T}$-local system by specifying the action of generators of $\check{\mathfrak{g}} \otimes H^{\bullet}(X)$ on $\overline{\mathfrak{p}}_*\IC_{\overline{\Bun}_B}$.

We construct the action of $\check{\mathfrak{h}} \otimes H^{\bullet}(X)$ as follows. Pullback along the evaluation map \[ X \times \Bun_T \longrightarrow \pt/T \] defines a homomorphism \[ \Sym(\mathfrak{h}^*[-2]) = H^{\bullet}(\pt/T) \longrightarrow H^{\bullet}(X) \otimes H^{\bullet}(\Bun_T). \] By adjunction we obtain a morphism $\mathfrak{h^*} \otimes H_{\bullet}(X)[-2] \longrightarrow H^{\bullet}(\Bun_T)$. Identifying $\mathfrak{h}^* \cong \check{\mathfrak{h}}$ and $H_{\bullet}(X)[-2] \cong H^{\bullet}(X)$, the latter using Poincar\'{e} duality, we obtain a morphism \[ \check{\mathfrak{h}} \otimes H^{\bullet}(X) \longrightarrow H^{\bullet}(\Bun_T). \] Then the action of $H^{\bullet}(\Bun_T)$ on $\omega_{\Bun_T}$ induces by functoriality the desired action of $\check{\mathfrak{h}} \otimes H^{\bullet}(X)$ on $\Eis_{!*}\omega_{\Bun_T} = \overline{\mathfrak{p}}_*\IC_{\overline{\Bun}_B}$.

Next we construct the action of $\check{\mathfrak{n}} \otimes H^{\bullet}(X)$. Denote by $\mathscr{U}(\check{\mathfrak{n}})$ the factorization algebra whose fiber at $\sum_i \mu_i x_i \in X^{\mu}$ is \[ \bigotimes_i U(\check{\mathfrak{n}})^{\mu_i}, \] where the superscript $\mu_i$ indicates the corresponding $\check{T}$-graded component. A result from \cite{FFKM} was restated as follows in \cite{BG2}: for any $\mu \in \Lambda^{\pos}$ there is a canonical morphism \[ \mathfrak{j}_{\geq \mu,!}(\mathscr{U}(\check{\mathfrak{n}})^{\mu} \boxtimes \IC_{\overline{\Bun}_B}) \longrightarrow \IC_{\overline{\Bun}_B}, \] which induces an isomorphism \[ \mathscr{U}(\check{\mathfrak{n}})^{\mu} \boxtimes \IC_{\Bun_B} \tilde{\longrightarrow} \mathfrak{j}_{=\mu}^!\IC_{\overline{\Bun}_B}. \]

Pushing forward to $\Bun_G$, we obtain a morphism \[ H^{\bullet}(X^{\mu},\mathscr{U}(\check{\mathfrak{n}})^{\mu}) \otimes \overline{\mathfrak{p}}_*\IC_{\overline{\Bun}_B} \longrightarrow \overline{\mathfrak{p}}_*\IC_{\overline{\Bun}_B}. \] The object $\mathscr{U}(\check{\mathfrak{n}})^{\mu}$ is concentrated in (perverse) cohomological degrees $\geq 1$, and if $\mu$ is a coroot then we have $H^1(\mathscr{U}(\check{\mathfrak{n}})^{\mu}) = \IC_X$. The resulting morphism $k_X \to \mathscr{U}(\check{\mathfrak{n}})^{\mu}$ induces \[ H^{\bullet}(X) \otimes \overline{\mathfrak{p}}_*\IC_{\overline{\Bun}_B} \longrightarrow H^{\bullet}(X^{\mu},\mathscr{U}(\check{\mathfrak{n}})^{\mu}) \otimes \overline{\mathfrak{p}}_*\IC_{\overline{\Bun}_B} \longrightarrow \overline{\mathfrak{p}}_*\IC_{\overline{\Bun}_B}, \] which defines the action of $\check{\mathfrak{n}}_{\mu} \otimes H^{\bullet}(X) \cong H^{\bullet}(X)$ on $\overline{\mathfrak{p}}_*\IC_{\overline{\Bun}_B}$.

Dually, for any $\mu \in \Lambda^{\pos}$ there is a canonical morphism \[ \IC_{\overline{\Bun}_B} \longrightarrow \mathfrak{j}_{\geq \mu,*}(\mathscr{U}^{\vee}(\check{\mathfrak{n}}^-)^{\mu} \boxtimes \IC_{\overline{\Bun}_B}), \] which induces an isomorphism \[ \mathfrak{j}_{=\mu}^*\IC_{\overline{\Bun}_B} \tilde{\longrightarrow} \mathscr{U}^{\vee}(\check{\mathfrak{n}}^-)^{\mu} \boxtimes \IC_{\Bun_B}. \] Here $\mathscr{U}^{\vee}(\check{\mathfrak{n}}^-)$ is by definition the Verdier dual of $\mathscr{U}(\check{\mathfrak{n}}^-)$.

Thus we obtain a morphism \[ \overline{\mathfrak{p}}_*\IC_{\overline{\Bun}_B} \longrightarrow H^{\bullet}(X^{\mu},\mathscr{U}(\check{\mathfrak{n}}^-)^{\mu})^{\vee} \otimes \overline{\mathfrak{p}}_*\IC_{\overline{\Bun}_B}, \] or by adjunction \[ H^{\bullet}(X^{\mu},\mathscr{U}(\check{\mathfrak{n}}^-)^{\mu}) \otimes \overline{\mathfrak{p}}_*\IC_{\overline{\Bun}_B} \longrightarrow \overline{\mathfrak{p}}_*\IC_{\overline{\Bun}_B}. \] If $\mu$ is a coroot, then as before we have a morphism $H^{\bullet}(X) \to H^{\bullet}(X^{\mu},\mathscr{U}(\check{\mathfrak{n}}^-)^{\mu})$, which defines the action of $\check{\mathfrak{n}}^-_{-\mu} \otimes H^{\bullet}(X) \cong H^{\bullet}(X)$ on $\overline{\mathfrak{p}}_*\IC_{\overline{\Bun}_B}$.

\subsection{} Fix a coroot $\mu$. Recall the subquotient $\mathscr{M}$ of $\Psi(\mathscr{W})$ from Lemma \ref{threestep}. The action of $\check{\mathfrak{g}}$ on $\overline{\mathfrak{p}}_*\IC_{\overline{\Bun}_B}$ relates to our problem through the following key lemma, whose proof will occupy this subsection.

\begin{lemma}

The sheaf $\mathscr{M}$ does not descend to $\overline{\Bun}_{N^{\omega}}/\G_m$.

\label{desclem}

\end{lemma}

First, observe that $\mathscr{M}_0$ and $\mathscr{M}/\mathscr{M}_{-1}$ descend to $\overline{\Bun}_{N^{\omega}}/T$ and hence to $\overline{\Bun}_{N^{\omega}}/\G_m$, being subquotients of $\mathfrak{j}_{=0,!}\IC_{\Bun_{N^{\omega}}}$ and $\mathfrak{j}_{=0,*}\IC_{\Bun_{N^{\omega}}}$ respectively. The obstruction to descent of $\mathscr{M}$ to $\overline{\Bun}_{N^{\omega}}/\G_m$ is the resulting composition
\begin{equation}
\Delta^{\mu}_*\IC_{(X \times_{\Bun_T} \overline{\Bun}_B)/\G_m} \longrightarrow \IC_{\overline{\Bun}_{N^{\omega}}/\G_m}[1] \longrightarrow \Delta^{\mu}_*\IC_{(X \times_{\Bun_T} \overline{\Bun}_B)/\G_m}[2].
\label{descobgm}
\end{equation}
Similarly, the obstruction to its descent to $\overline{\Bun}_{N^{\omega}}/T$ is the composition
\begin{equation}
\Delta^{\mu}_*\IC_{(X \times_{\Bun_T} \overline{\Bun}_B)/T} \longrightarrow \IC_{\overline{\Bun}_{N^{\omega}}/T}[1] \longrightarrow \Delta^{\mu}_*\IC_{(X \times_{\Bun_T} \overline{\Bun}_B)/T}[2].
\label{descobt}
\end{equation}

Denote by $\mathring{\Delta}^{\mu} : X \times_{\Bun_T} \Bun_B \to \overline{\Bun}_{N^{\omega}}$ the locally closed embedding given by composing $\Delta^{\mu}$ with the open embedding $X \times_{\Bun_T} \Bun_B \to X \times_{\Bun_T} \overline{\Bun}_B$. Composition with the canonical morphisms \[ \mathring{\Delta}^{\mu}_!\IC_{(X \times_{\Bun_T} \Bun_B)/T} \to \Delta^{\mu}_*\IC_{(X \times_{\Bun_T} \overline{\Bun}_B)/T} \] and \[ \Delta^{\mu}_*\IC_{(X \times_{\Bun_T} \overline{\Bun}_B)/T} \to \mathring{\Delta}^{\mu}_*\IC_{(X \times_{\Bun_T} \Bun_B)/T} \] gives
\begin{equation}
\End(\Delta^{\mu}_*\IC_{(X \times_{\Bun_T} \overline{\Bun}_B)/T}) \longrightarrow H^{\bullet}(X \times \pt/T),
\label{cohomap}
\end{equation}
since the map $(X \times_{\Bun_T} \Bun_B)/T \to X \times \pt/T$ induces an isomorphism on cohomology.

\begin{lemma}

The image of the endomorphism (\ref{descobt}) under (\ref{cohomap}) is \[ -1 \otimes h_{\mu} \in H^0(X) \otimes \mathfrak{h}^* \subset H^2(X \times \pt/T). \]

\label{desclemt}

\end{lemma}

\begin{proof}

Theorem 5.1.5 in \cite{BG1} says that $\IC_{\overline{\Bun}_B}$ is ULA over $\Bun_T$, which implies that the $!$-restriction of $\IC_{\overline{\Bun}_B}$ to $\overline{\Bun}_{N^{\omega}}/T$ is $\IC_{\overline{\Bun}_{N^{\omega}}/T}[\dim T - \dim \Bun_T]$. It follows that the $!$-restriction of $\Delta^{\mu}_*\IC_{X \times \overline{\Bun}_B}$ to $\overline{\Bun}_{N^{\omega}}/T$ is a shift of $\Delta^{\mu}_*\IC_{(X \times_{\Bun_T} \overline{\Bun}_B)/T}$, where we abusively write $\Delta^{\mu} : X \times \overline{\Bun}_B \to \overline{\Bun}_B$ for the similarly-defined finite map. This gives rise to a commutative square \[
\begin{tikzcd}[cramped]
\End(\Delta^{\mu}_*\IC_{X \times \overline{\Bun}_B}) \ar[d] \ar[r] & H^{\bullet}(X \times \Bun_T) \ar[d] \\
\End(\Delta^{\mu}_*\IC_{(X \times_{\Bun_T} \overline{\Bun}_B)/T}) \ar[r] & H^{\bullet}(X \times \pt/T),
\end{tikzcd} \]
where the upper horizontal arrow is defined similarly to (\ref{cohomap}) and the right vertical arrow is $\id_X$ times restriction along $\rho(\omega) : \pt/T \to \Bun_T$. The previous subsection implies that $\mathscr{M}_0$ and $\mathscr{M}/\mathscr{M}_{-1}$ extend to $\overline{\Bun}_B$, giving rise to a morphism
\begin{equation}
\Delta^{\mu}_*\IC_{X \times \overline{\Bun}_B} \longrightarrow \IC_{\overline{\Bun}_B}[1] \longrightarrow \Delta^{\mu}_*\IC_{X \times\overline{\Bun}_B}[2]
\label{obbunb}
\end{equation}
which restricts to (\ref{descobt}) on $\overline{\Bun}_{N^{\omega}}/T$. Thus it suffices to show that the image of (\ref{obbunb}) in $H^2(X \times \Bun_T)$ restricts to $-1 \otimes h_{\mu}$ along $\id_X \times \rho(\omega)$.

Observe that \[ \End(\Delta^{\mu}_*\IC_{X \times \overline{\Bun}_B}) = \Hom(\Delta^{\mu}_*(k_X \boxtimes \IC_{\overline{\Bun}_B}),\Delta^{\mu}_*(\omega_X \boxtimes \IC_{\overline{\Bun}_B}))[-2], \] so $\overline{\mathfrak{p}}_*$ induces a morphism \[ \End(\Delta^{\mu}_*\IC_{X \times \overline{\Bun}_B}) \longrightarrow \Hom(H^{\bullet}(X) \otimes \overline{\mathfrak{p}}_*\IC_{\overline{\Bun}_B},H_{\bullet}(X) \otimes \overline{\mathfrak{p}}_*\IC_{\overline{\Bun}_B})[-2]. \] Composing with the canonical map $k \to H^{\bullet}(X)$ and its dual $H_{\bullet}(X) \to k$, we obtain
\begin{equation}
\End(\Delta^{\mu}_*\IC_{X \times \overline{\Bun}_B}) \longrightarrow \End(\overline{\mathfrak{p}}_*\IC_{\overline{\Bun}_B})[-2].
\label{eismap}
\end{equation}
By construction, the image of (\ref{descobt}) under (\ref{eismap}) coincides with the action of $f_{\mu}e_{\mu} \in U(\mathfrak{\check{g}})$. By composing with the morphisms $\mathfrak{p}_!\IC_{\Bun_B} \to \overline{\mathfrak{p}}_*\IC_{\overline{\Bun}_B}$ and $\overline{\mathfrak{p}}_*\IC_{\overline{\Bun}_B} \to \mathfrak{p}_*\IC_{\Bun_B}$, we obtain
\begin{equation}
\End(\overline{\mathfrak{p}}_*\IC_{\overline{\Bun}_B}) \longrightarrow \Hom(\mathfrak{p}_!\IC_{\Bun_B},\mathfrak{p}_*\IC_{\Bun_B}).
\label{eismap2}
\end{equation}
Note that (\ref{eismap2}) annihilates the endomorphism of $\overline{\mathfrak{p}}_*\IC_{\overline{\Bun}_B}$ given by the action of $e_{\mu}f_{\mu}$, since it factors through a sheaf supported on the boundary. Proposition 4.9 in \cite{FFKM} says that the relation $[e_{\mu},f_{\mu}] = h_{\mu}$ holds in $\End(\overline{\mathfrak{p}}_*\IC_{\overline{\Bun}_B})$, which implies that the images of $f_{\mu}e_{\mu}$ and $-h_{\mu}$ under (\ref{eismap2}) coincide.

Now consider the commutative square \[
\begin{tikzcd}[cramped]
\End(\Delta^{\mu}_*\IC_{X \times \overline{\Bun}_B}) \ar[d, "(\ref{eismap})"'] \ar[r, "(\ref{cohomap})"] & H^{\bullet}(X \times \Bun_T) \ar[d] \\
\End(\overline{\mathfrak{p}}_*\IC_{\overline{\Bun}_B})[-2] \ar[r, "(\ref{eismap2})"] & \Hom(\mathfrak{p}_!\IC_{\Bun_B},\mathfrak{p}_*\IC_{\Bun_B})[-2],
\end{tikzcd} \]
where the right vertical morphism is the composition
\begin{align*}
H^{\bullet}(X \times \Bun_T) = &\Hom(\mathring{\Delta}_!^{\mu}\IC_{X \times \Bun_B},\mathring{\Delta}_*^{\mu}\IC_{X \times \Bun_B}) \\ \longrightarrow &\Hom(H^{\bullet}(X) \otimes \mathfrak{p}_!\IC_{\Bun_B},H_{\bullet}(X) \otimes \mathfrak{p}_*\IC_{\Bun_B})[-2] \\ \longrightarrow &\Hom(\mathfrak{p}_!\IC_{\Bun_B},\mathfrak{p}_*\IC_{\Bun_B})[-2].
\end{align*}
Note that $\id_X \times \rho(\omega) : X \times \pt/T \to X \times \Bun_T$ admits a canonical retraction, given by the projection $X \times \Bun_T \to X$ and the evaluation map $X \times \Bun_T \to \pt/T$. So far we have shown that the image of (\ref{descobt}) under the resulting composition \[ R^2\End(\Delta^{\mu}_*\IC_{(X \times_{\Bun_T} \overline{\Bun}_B)/T}) \longrightarrow H^2(X \times \pt/T) \longrightarrow H^2(X \times \Bun_T) \longrightarrow R^0\Hom(\mathfrak{p}_!\IC_{\Bun_B},\mathfrak{p}_*\IC_{\Bun_B}) \] agrees with the image of $-1 \otimes h_{\mu}$ under \[ H^2(X \times \pt/T) \longrightarrow H^2(X \times \Bun_T) \longrightarrow R^0\Hom(\mathfrak{p}_!\IC_{\Bun_B},\mathfrak{p}_*\IC_{\Bun_B}), \] so it suffices to show that the latter composition is injective.

We have $H^2(X \times \pt/T) = \check{\mathfrak{h}} \oplus H^2(X)$, and we have already used the fact that for each $\lambda \in \Lambda$, an element $h \in \check{\mathfrak{h}}$ maps to $\langle h,\lambda \rangle$ times the canonical morphism $\mathfrak{p}_!\IC_{\Bun_B^{\lambda}} \to \mathfrak{p}_*\IC_{\Bun_B^{\lambda}}$. One checks that the canonical generator of $H^2(X)$ maps to canonical map $\mathfrak{p}_!\IC_{\Bun_B} \to \mathfrak{p}_*\IC_{\Bun_B}$ itself, which proves the desired injectivity.

\end{proof}

Note that Lemma \ref{desclemt} already implies that $\mathscr{M}$ does not descend to $\overline{\Bun}_{N^{\omega}}/T$, since $h_{\mu} \neq 0$.

\begin{proof}[Proof of Lemma \ref{desclem}]

The morphism (\ref{descobgm}) induces an element of $H^2(X \times \pt/\G_m)$ in the same way that (\ref{descobt}) gives rise to $-1 \otimes h_{\mu} \in H^2(X \times \pt/T)$. Moreover, these constructions fit into a commutative square \[
\begin{tikzcd}[cramped]
\End(\Delta^{\mu}_*\IC_{(X \times_{\Bun_T} \overline{\Bun}_B)/T}) \ar[d] \ar[r] & H^{\bullet}(X \times \pt/T) \ar[d] \\
\End(\Delta^{\mu}_*\IC_{(X \times_{\Bun_T} \overline{\Bun}_B)/\G_m}) \ar[r] & H^{\bullet}(X \times \pt/\G_m)
\end{tikzcd} \]
where the vertical morphisms are induced by $\gamma$, and in particular (\ref{descobt}) maps to (\ref{descobgm}) along the left vertical morphism. The image of $h_{\mu}$ under $H^2(\pt/T) \to H^2(\pt/\G_m) = k$ is the positive integer $\langle h_{\mu},\gamma \rangle$, so the lemma follows.

\end{proof}

\subsection{} We need another, more elementary lemma. Fix $\mu \in \Lambda^{\pos}$ and $\mathfrak{k} \in \Kost(\mu)$ given by $\mu = \sum n_{\beta} \beta$.

\begin{lemma}

If $\mathscr{L}$ is a nonconstant simple summand of $\boxtimes_{\beta \in R^+} \mathscr{P}_{n_{\beta}}$, then \[ \Delta^!\iota^{\mathfrak{k}}_*\mathscr{L} = 0 = \Delta^*\iota^{\mathfrak{k}}_*\mathscr{L}. \]

\label{plclean}

\end{lemma}

\begin{proof}

This follows from the fact that the local system on $X^{(n)}_{\disj}$ associated to a nontrivial irreducible $\Sigma_n$-representation extends cleanly over the main diagonal.

\end{proof}

\begin{proof}[Second proof of Theorems \ref{main} and \ref{zmain}]

We proceed as in the first proof, until we have constructed the isomorphism (\ref{zmaineq}) over $Z^{\mu} \setminus \Delta(X)$ and reduced to the case that $\mu$ is a coroot. We showed that in this case $\Psi( \, \! \mathscr{W}_{Z^{\mu}})$ contains $\Delta_*\IC_X$ with multiplicity two but has no other subquotients supported on the main diagonal (in particular, the two sides of (\ref{zmaineq}) agree in the Grothendieck group). Either $\mathfrak{sl}_2$ acts trivially on both copies of $\Delta_*\IC_X$ or they have weights $1$ and $-1$. In order to show that (\ref{zmaineq}) extends to $Z^{\mu}$ we need to rule out the first case, and then show that $\gr \Psi( \, \! \mathscr{W}_{Z^{\mu}})$ is semisimple.

We must rule out the possibility that $\mathfrak{sl}_2$ acts trivially on the subquotient $\mathscr{M}$ from Lemma \ref{threestep}. Since the monodromy endomorphism of $\Psi( \, \! \mathscr{W}_{Z^{\mu}})$ is the obstruction to $\G_m$-equivariance, this would imply that $\mathscr{M}$ is $\G_m$-equivariant, contradicting Lemma \ref{threestep}.

Now we finish the proof that $\gr \Psi( \, \! \mathscr{W}_{Z^{\mu}})$ is semisimple. Using the $\mathfrak{sl}_2$-action, it decomposes into the direct sum of its isotypic components, indexed by the irreducible $\mathfrak{sl}_2$-representations. The previous paragraph says that the two copies of $\Delta_*\IC_X$ are subquotients of the $\std$-isotypic component. Thus the other isotypic components have no subquotients supported on the main diagonal, so they are the same as the corresponding isotypic components on the right hand side of (\ref{zmaineq}) and in particular are semisimple. We will show that for any simple subquotient $\mathscr{L} \neq \Delta_*\IC_X$ of the $\std$-isotypic component we have \[ \Ext^1(\Delta_*\IC_X,\mathscr{L}) = 0 = \Ext^1(\mathscr{L},\Delta_*\IC_X), \] from which it follows that $\IC_X^{\oplus 2}$ is a direct summand of the $\std$-isotypic component. Since the the other summand has no subquotients supported on the main diagonal, it is semisimple by the induction hypothesis.

We will show that $H^i(\Delta^!\mathscr{L}) = 0$ for $i \leq 1$, which implies that $\Ext^1(\Delta_*\IC_X,\mathscr{L}) = 0$. The other vanishing follows by applying Verdier duality: since $\gr \Psi( \, \! \mathscr{W}_{Z^{\mu}})$ is Verdier self-dual, so is its $\std$-isotypic component. Observe that $\mathscr{L}$ has the form $^{\prime}\mathfrak{j}^{\mu}_{=\nu,!*} \ \! ^{\prime}\mathfrak{m}^{\mu,\Delta}_{\nu}\mathscr{F}$ for some $0 < \nu \leq \mu$ and a simple summand $\mathscr{F}$ of $\mathscr{P}^{\nu}$. The case $\nu = 0$ is excluded because then $\mathscr{L} = \IC_{Z^{\mu}}$ has weight $0$.

First suppose that $\nu < \mu$. Then we have \[ \Delta^! \, \! ^{\prime}\mathfrak{j}^{\mu}_{=\nu,!*} \ \! ^{\prime}\mathfrak{m}^{\mu,\Delta}_{\nu}\mathscr{F} \tilde{\longrightarrow} \Delta^!\mathscr{F} \otimes^! \Delta^!\IC_{Z^{\mu - \nu}}. \] Since $\Delta^!\IC_{Z^{\mu - \nu}}$ is concentrated in cohomological degrees $\geq 1$ (see \cite{BG2}) and $\Delta^!\mathscr{F}$ is concentrated in degrees $\geq 0$, and both complexes have lisse cohomology sheaves, their $!$-tensor product is concentrated in degrees $\geq 2$ as desired.

Finally, we address the case $\nu = \mu$, where $\mathscr{F} = \mathscr{L}$. By Lemma \ref{plclean}, we can assume $\mathscr{L}$ is a summand of $\add_*\IC_{\prod X^{(n_{\beta})}}$ for some Kostant partition $\mu = \sum n_{\beta} \beta$. If $\sum n_{\beta} \geq 3$ then the claim follows by base change. By assumption $\mathscr{L} \neq \Delta_*\IC_X$, so $\sum n_{\beta} > 1$. This leaves only the case $\sum n_{\beta} = 2$, and since $\mu$ is a coroot the only possibility is that $\mu = \beta_1 + \beta_2$ is a sum of two distinct coroots. As shown in Lemma \ref{sumsimpex}, in this case the $\std$-isotypic component of $\gr \Psi(\mathscr{W}_{Z^{\mu}})$ is just $\IC_X^{\oplus 2}$.

\end{proof}

\label{mainproof2}


\begin{thebibliography}{10}

\bibitem{AG}
D. Arinkin and D. Gaitsgory: Asymptotics of geometric Whittaker coefficients

\bibitem{B}
A. Beilinson: How to glue perverse sheaves, in K-theory, arithmetic, and geometry, Lect. Notes Math., vol. 1289. Springer, Berlin (1987).

\bibitem{BG1}
A. Braverman and D. Gaitsgory: Geometric Eisenstein series, Invent. Math. \textbf{150}, 287-384 (2002).

\bibitem{BG2}
A. Braverman and D. Gaitsgory: Deformations of local systems and Eisenstein series, Geometric and functional analysis \textbf{17}(6), 1788-1850 (2008).

\bibitem{C1}
J. Campbell: The big projective module as a nearby cycles sheaf, J. Sel. Math. New Ser. (2016). 

\bibitem{C2}
J. Campbell: A resolution of singularities for Drinfeld's compactification by stable maps, arXiv:1606.01518 (2016).

\bibitem{ENV}
M. Emerton, D. Nadler, and K. Vilonen: A geometric Jacquet functor, Duke Mathematical Journal \textbf{125}(2), 267-278 (2004).

\bibitem{FFKM}
B. Feigin, M. Finkelberg, A. Kuznetsov, and I. Mirkovic: Semiinfinite flags II, arXiv:alg-geom/9711009.

\bibitem{FGV}
E. Frenkel, D. Gaitsgory, and K. Vilonen: Whittaker patterns in the geometry of moduli spaces of bundles on curves," Annals of Mathematics \textbf{153}(3), 699-748 (2001).

\bibitem{G}
D. Gaitsgory: Outline of the proof of the geometric Langlands conjecture for GL (2), arXiv:1302.2506 (2013).

\bibitem{M}
T. Mochizuki: Mixed twistor D-modules, arXiv:1104.3366 (2013).

\bibitem{R}
S. Raskin: Chiral principal series categories I: finite-dimensional calculations, available at http://math.mit.edu/~sraskin/cpsi.pdf (2016).

\bibitem{S1}
S. Schieder: The Drinfeld-Lafforgue-Vinberg degeneration I: Picard-Lefschetz oscillators, arXiv:1411.4206 (2014).

\bibitem{S2}
S. Schieder: Monodromy and Vinberg fusion for the principal degeneration of the space of $G$-bundles, arXiv:1701.01898 (2017).

\end{thebibliography}
\end{document}